\documentclass[12pt]{article}
\usepackage[utf8]{inputenc}
\usepackage{amsmath, dsfont, amssymb, bm, amsthm}
\usepackage{graphicx, epstopdf, geometry, color, subfigure, enumitem}
\usepackage{float}

\def\R{\mathbb{R}}
\def\N{\mathbb{N}}

\newcommand{\ds}{\displaystyle}

\newcommand{\T}{\mathcal{T}}

\newtheorem{theorem}{Theorem}[section]
\newtheorem{lemma}[theorem]{Lemma}

\newtheorem{corollary}[theorem]{Corollary}
\newtheorem{proposition}[theorem]{Proposition}

\def\lp{\left(}
\def\rp{\right)}

\definecolor{aquamarine}{rgb}{0.13, 0.68, 0.8}

\usepackage[english]{babel}
\usepackage[babel=true]{csquotes} 
\usepackage{varioref} 
\usepackage{hyperref} 
\hypersetup{
    colorlinks=true,
    linkcolor=blue,
}
\newcommand{\pth}[1]{\left(#1\right)}
\newcommand{\cro}[1]{\left[#1\right]}

\newcommand{\eps}{\varepsilon}
\newcommand{\toweak}{\rightharpoonup}

\newcommand{\mcl}{\mathcal{L}}
\newcommand{\mcc}{\mathcal{C}}

\begin{document}

\title{\bf{Periodic KPP equations: new insights into persistence, spreading, and the role of advection}} 
\date{}
\author{N. Boutillon$^{\footnotesize\hbox{a,b}}$, F. Hamel$^{\footnotesize\hbox{a}}$, L. Roques$^{\footnotesize\hbox{b}}$ \\
\\
\footnotesize{$^{\hbox{a }}$Aix Marseille Univ, CNRS, I2M, Marseille, France}\\
\footnotesize{$^{\hbox{b }}$INRAE, BioSP, 84914, Avignon, France}\\
}

\maketitle

\begin{abstract}
We focus on the persistence and spreading properties for a heterogeneous Fisher-KPP equation with advection. After reviewing the different notions of persistence and spreading speeds, we focus on the effect of the direction of the advection term, denoted by $b$. First, we prove that changing $b$ to $-b$ can have an effect on the spreading speeds and the ability of persistence. Next, we provide a class of relationships between the intrinsic growth term $r$ and the advection term $b$ such that changing $b$ to $-b$ does {\emph not} change the spreading speeds and the ability of persistence. We briefly mention the cases of slowly and rapidly varying environments, and bounded domains. Lastly, we show that in general, the spreading speeds are not controlled by the ability of persistence, and conversely. However, when there is no advection term, the spreading speeds are controlled by the ability of persistence,  though the converse still does not hold.
\end{abstract}

\noindent \textbf{Keywords:} KPP equation; advection term; principal eigenvalue; persistence; spreading speeds.

\medskip

\noindent \textbf{MSC 2020:} 35K57, 35B40, 92D25, 35Q92.

\tableofcontents


\section{Introduction \label{sec:intro}}

In this article, we focus on the persistence and spreading properties of an invasive species living in a heterogeneous periodic environment. We address the following questions:
\begin{enumerate}
\item Does the direction of an advection force influence the ability of persistence and the leftward and rightward spreading speeds of the species?
\item Does a better ability of persistence imply higher spreading speeds of the species?
\end{enumerate}

To answer these questions, we work in the formalism of reaction-diffusion equations in a heterogeneous periodic environment, with an advection term:
\begin{equation} \label{eq:KPP}
\partial_t u = \partial_{xx} u - \partial_x (b(x) u) + f(x,u), \quad t>0, \ x \in \mathbb{R},
\end{equation}
with periodic coefficients in $x$ and unknown function $u(t,x)\ge0$. Equations of this type are commonly considered to model the dynamical properties and propagation phenomena of a population, often assumed to be initially compactly supported, in an unbounded heterogeneous medium~\cite{BerHam02,CanCos03,RoqL13,Shikaw97,Wei02,Xin00,Zha17}. The first objective of this work is to analyze the effect of replacing $b$ with $-b$ in the advection term. In particular, a natural question is whether this modification affects the ability of persistence of the population. Due to the periodic nature of the problem, we will see that the answer is not straightforward and cannot be obtained through trivial reasoning. The second objective of this work is to study the relationship between the ability of persistence and the spreading speeds of the population, when the initial condition has a compact support.

Before proceeding further, let us clarify the assumptions and precisely define the concepts of persistence and spreading.

\subsubsection*{Main assumptions}

Throughout the article, the function $b$ is assumed to be $1$-periodic and of class $\mathcal{C}^{1,\alpha}(\R)$, for some $\alpha\in(0,1)$. We assume that the function $f: \mathbb{R} \times \mathbb{R}_+ \to \mathbb{R},\ (x,u)\mapsto f(x,u)$ is of class $\mathcal{C}^{0,\alpha}_{\mathrm{loc}}(\R\times\R_+)$, and that $\frac{\partial f}{\partial u}$ exists and is of class $\mathcal{C}^{0,\alpha}_{\mathrm{loc}}(\R\times\R_+)$. We also assume that $f$ is $1$-periodic with respect to $x$ and that
\begin{equation}\label{f0}
f(x,0)=0\ \hbox{ for all }x\in\mathbb{R}.
\end{equation}
We define, for $x\in\R$,
\begin{equation}
    r(x) := \left. \frac{\partial f}{\partial u}(x,u) \right|_{u=0}.
\end{equation}
The function $r$ is then $1$-periodic and of class $\mathcal{C}^{0,\alpha}(\R)$. Last, we assume, throughout the paper, that the following strong variant of the KPP (standing for Kolmogorov-Petrovsky-Piskunov) assumption holds:
\begin{equation}\label{f1}
\forall\,x \in \mathbb{R},\ u \mapsto \frac{f(x,u)}{u} \text{ is decreasing in } u > 0
\end{equation}
and
\begin{equation}\label{f2}
\exists\,M \ge 0,\ \forall\,u \ge M,\ \forall\,x \in \mathbb{R},\ f(x,u)-b'(x)u \le 0.
\end{equation}
These assumptions, together with the other one~\eqref{eq:persist_cond} below on the linear instability of the trivial state, will ensure the existence and uniqueness of a positive $1$-periodic steady state to~\eqref{eq:KPP}. When we see $u$ as a population density, the coefficient~$r(x)$ is interpreted as the intrinsic growth rate of the population. Examples of functions~$f$ satisfying the above assumptions include $f(x,u) = r (x)\, u (1-u)$ or $f(x,u) = u(r(x)-\gamma(x)u)$, where~$r$ and~$\gamma$ are $\mathcal{C}^{0,\alpha}(\R)$ $1$-periodic positive functions.

\subsubsection*{Characterization of persistence and spreading}

For this kind of reaction-diffusion equation where the KPP assumption is satisfied, it is well-known that the ability of persistence, as well as the asymptotic spreading speeds, for solutions that are initially bounded, continuous, non-negative and compactly supported, can be characterized by principal eigenvalues of elliptic operators with periodicity conditions. Namely, for $\lambda \in \mathbb{R}$, we introduce the operator:
\begin{equation}\label{defLlambda}
\mathcal{L}_\lambda[r;b] \, : \, \psi \mapsto \psi'' + (b+2 \lambda) \psi' + \left(\lambda^2 + \lambda \, b  + r \right)\psi,
\end{equation}
acting on $1$-periodic $\mathcal{C}^2$ functions. By the Krein-Rutman theorem, for each $\lambda\in\R$, there exists a unique $k_{\lambda}[r;b]\in\R$ such that there is a $1$-periodic function $\psi\in\mathcal{C}^{2,\alpha}(\R)$ satisfying
\begin{equation}\label{klambda}
\mathcal{L}_\lambda[r;b]\,\psi =k_{\lambda}[r;b]\,\psi\ \hbox{ in }\R,\ \ \ \psi>0\ \hbox{ in }\R.
\end{equation}
We say that $k_{\lambda}[r;b]$ is the principal eigenvalue of $\mathcal{L}_\lambda[r;b]$. We point out that $\mcl_{0}[r;b]$ is the adjoint of the operator
\[\mcl_0[r;b]^*:\psi \mapsto \psi'' -  (b\psi)' + r\psi,\]
still acting on $1$-periodic $\mathcal{C}^2$ functions, and which is obtained from the linearization of the right-hand side of~\eqref{eq:KPP} around the trivial state~$0$. In particular, $\mcl_{0}[r;b]$ and $\mcl_0[r;b]^*$ have the same principal eigenvalue $k_0[r;b]$.

The concept of \emph{persistence} refers to a scenario where, given any function~$u_0$ that is bounded, continuous, non-negative, and not identically equal to zero, the solution~$u$ to the Cauchy problem~\eqref{eq:KPP} with $u(0,\cdot)=u_0$ satisfies
\begin{equation}\label{defpersistence}
{\limsup_{t\to+\infty}\|u(t,\cdot)\|_{L^\infty(\R)}>0}.
\end{equation}
The concept of persistence is directly opposed to the concept of \emph{extinction}, which means that the solution~$u$ converges uniformly to $0$ at $t\to+\infty$. Actually, due to~\eqref{f0} and~\eqref{f1}, it follows from~\cite[Theorem~2.1 and Remark~3.1]{BerHamRoq05a} that, if $k_0[r;b]\le0$, then any solution~$u$ of~\eqref{eq:KPP} with a bounded continuous non-negative initial condition~$u_0$ goes to extinction as $t\to+\infty$. On the other hand, the condition
\begin{equation} \label{eq:persist_cond}
k_{0}[r;b] > 0,
\end{equation}
together with~\eqref{f0} and~\eqref{f1}-\eqref{f2}, warrants the existence and uniqueness of a positive $1$-periodic steady state $p>0$ to~\eqref{eq:KPP}, by~\cite[Theorem~2.1 and Remark~3.1]{BerHamRoq05a} and~\cite[Lemma~3.5, Proposition~6.3 and Theorem~6.5]{Berhamros07}. Furthermore, when there is no advection ($b=0$), condition~\eqref{eq:persist_cond} implies that $u(t,\cdot)\to p$ as $t\to+\infty$ locally uniformly in $\R$, for any continuous non-negative compactly supported initial condition $u_0\not\equiv0$ (by~\cite[Theorem~2.6]{BerHamRoq05a}). In the general case $b\not\equiv0$, the condition for local convergence to $p$ must be adapted. Namely, if $k_0[r-b'/2-b^2/4;0]>0$, then~\eqref{eq:persist_cond} holds and $u(t,\cdot)\to p$ as $t\to+\infty$ locally uniformly in $\R$ (see Proposition~\ref{ppn:persistence_general_b} in Appendix~\ref{s:proof_persistence_general_b}).

More generally speaking, by~\cite[Theorems~2.2,~2.3]{Wei02}, condition~\eqref{eq:persist_cond} implies persistence of the solution~$u$ to~\eqref{eq:KPP}, in the sense of~\eqref{defpersistence}, with any continuous non-negative compactly supported initial condition $u_0\not\equiv0$.\footnote{Persistence then holds even if the initial condition is not compactly supported, by the comparison principle.} Moreover, under condition~\eqref{eq:persist_cond}, there are two {\it spreading speeds}~$c_\pm[r;b]\in\R$,
defined by
\begin{equation}\label{defcpm}\left\{\begin{array}{l}
c_+[r;b]:=\inf\big\{c\in\R \ /\ \forall\,w>c,\, u(t,\, wt)\to 0 \text{ as $t\to+\infty$}\big\},\vspace{3pt}\\
c_-[r;b]:=\inf\big\{c\in\R \ /\ \forall\,w>c,\, u(t,\, -wt)\to 0 \text{ as $t\to+\infty$}\big\}.\end{array}\right.
\end{equation}
The two spreading speeds~$c_\pm[r;b]$ are independent of $u_0$, and satisfy $-c_-[r;b]<c_+[r;b]$, and
$$\forall\,w\le w'\in(-c_-[r;b],c_+[r;b]),\quad\max_{wt\le x\le w't}|u(t,x)-p(x)|\to0\hbox{ as $t\to+\infty$}.$$
When condition~\eqref{eq:persist_cond} is not fulfilled, that is, when $k_0[r;b]\le0$, then we set $c_\pm[r;b]:=-\infty$, which is coherent with~\eqref{defcpm} since $\|u(t,\cdot)\|_{L^\infty(\R)}\to0$ as $t\to+\infty$ in this case.

When the two speeds $c_-[r;b]$ and $c_+[r;b]$ are positive, $c_+[r;b]$ is known as the \emph{rightward spreading speed} and may be interpreted as follows: an observer moving rightwards at a speed greater than $c_+[r;b]$ will observe the population density tend to zero. Conversely, an observer moving rightwards at a nonnegative speed smaller than $c_+[r;b]$ will observe a population density that remains above a positive level.
Analogously, $c_-[r;b]$ is called the \emph{leftward spreading speed} and can be interpreted in the same way with \enquote{left} instead of \enquote{right}. The positivity of both speeds $c_-[r;b]$ and $c_+[r;b]$ is guaranteed for instance when $b=0$ and $k_0[r;0]>0$ and, more generally speaking, there even holds
\begin{equation}\label{c-=c+}
c_-[r;b]=c_+[r;b]>0\quad \hbox{ when }\int_0^1b=0\hbox{ and $k_0[r,b]>0$},
\end{equation}
see~\cite[Proposition~3.9]{BKR25}. We point out that such an equality does not hold for bistable-type equations, for which the leftward and rightward spreading speeds are not equal in general, even without advection term~\cite{DinGil21}. 

The concept of spreading speed is somewhat subtle for general periodic advection terms $b$, as in our setting. In particular, we can have $c_-[r;b]<0$ or $c_+[r;b]<0$. Indeed, for instance, if $r$ is equal to a positive constant $r_0$ and if $b$ is equal to a constant~$b_0$, then
\begin{equation}\label{eq:c+-}
c_+[r_0;b_0]=b_0+2\sqrt{r_0}\ \hbox{ and }\ c_-[r_0;b_0]=-b_0+2\sqrt{r_0},
\end{equation}
as immediately follows from~\cite{AroWei78}. In one of the situations $c_-[r;b]<0$ or $c_+[r;b]<0$, it means that a compactly supported initial datum gives rise to a solution to~\eqref{eq:KPP} which does not propagate to the left (if $c_-[r;b]<0$) or to the right (if $c_+[r;b]<0$), see Figure~\ref{fig:explanation_c_pm}. 
\begin{figure}[h]
\begin{center}
\includegraphics[width=12cm]{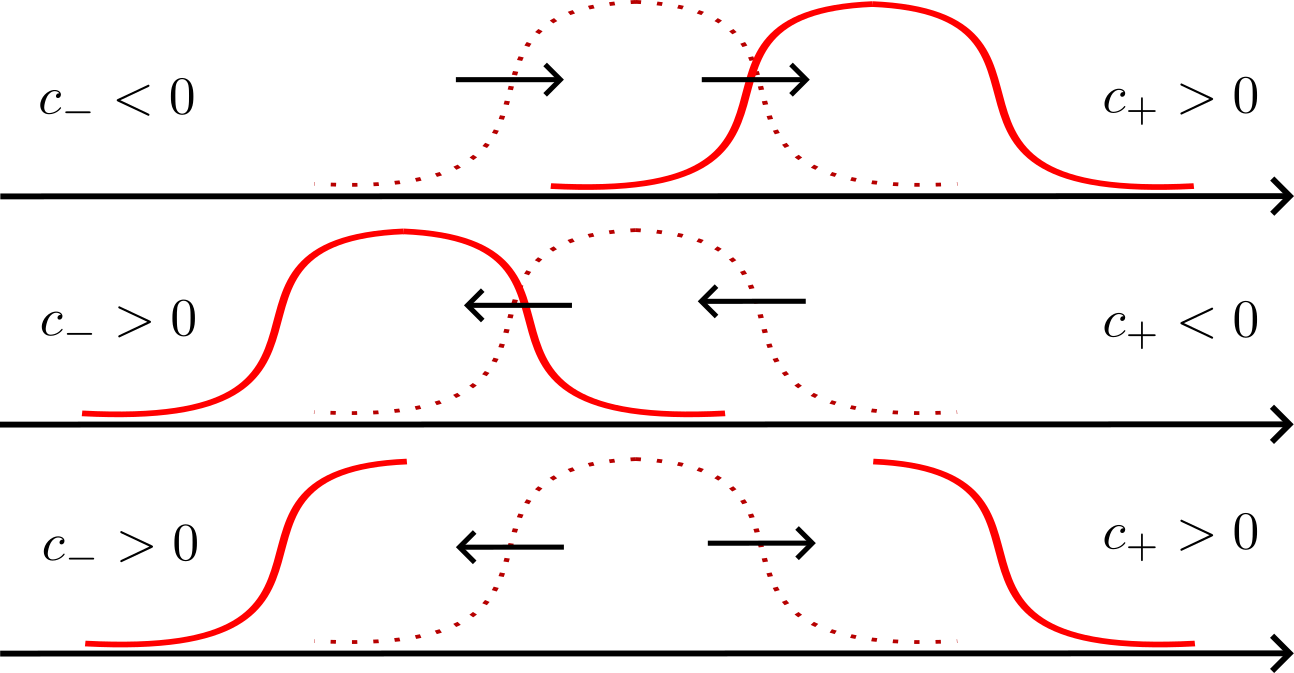}
\end{center}
\caption{\label{fig:explanation_c_pm}Three different situations. In each case, the dotted plot is the initial configuration and the full-line plot is taken at a later time.}
\end{figure}

As a matter of fact, the two formulas~\eqref{eq:c+-} when~$r=r_0$ and~$b=b_0$ are constant are particular cases of more general so-called Freidlin-Gärtner formulas. More precisely, by~\cite[Theorems~2.4,~2.5]{Wei02} and~\cite[Theorem~1.14]{Berhamnad08}, the spreading speeds~$c_+[r;b]$ and~$c_-[r;b]$ are characterized by:
\begin{equation}\label{eq:garfre}
c_+[r;b] = \inf_{\lambda > 0} \frac{k_{\lambda}[r;b]}{\lambda},\qquad c_-[r;b] = \inf_{\lambda > 0} \frac{k_{-\lambda}[r;b]}{\lambda},
\end{equation}
where we recall that, for $\lambda\in\R$, $k_\lambda[r;b]$ denotes the principal periodic eigenvalue of the operator $\mcl_\lambda[r;b]$ given in~\eqref{defLlambda}-\eqref{klambda}. Furthermore, as in the proof of~\cite[Proposition~5.7]{BerHam02}, the function $\lambda\mapsto k_\lambda[r;b]$ is convex in $\R$, whence continuous in $\R$, and, from~\eqref{klambda} applied at maximum or minimum points of a principal eigenfunction $\psi$, it follows that
\begin{equation}\label{kbornes}
\min_\R\,(\lambda^2+\lambda b+r)\le k_\lambda[r;b]\le\max_\R\,(\lambda^2+\lambda b+r).
\end{equation}
Therefore, $k_{\pm\lambda}[r;b]/\lambda\to+\infty$ as $\lambda\to+\infty$ and, together with~\eqref{eq:persist_cond}, it follows that both infima in~\eqref{eq:garfre} are reached. The Freidlin-Gärtner formulas~\eqref{eq:garfre}, which were first obtained in~\cite{GarFre79} in the case $b=0$, will be central to this work, as they imply that studying the principal eigenvalues of the operators $\mcl_{\lambda}[r;b]$ is sufficient to study the spreading speeds. We lastly mention that other formulas of the spreading speeds have been given for equations in higher dimensional media~\cite{BerHamNad05d,Fre84,GarFre79,Wei02} or with other types of reactions~$f$~\cite{Ros17} .


\section{Main results}

Section~\ref{sec21} is devoted to the study of the effect of changing $b$ to $-b$ on the ability of persistence, measured by $k_0[r;b]$, and the spreading speeds $c_\pm[r;b]$. Section~\ref{sec22} is concerned with the link between the ability of persistence and the spreading speeds, and in particular on the non-monotone relationship between these two concepts. Some of the results focus on $c_+[r;b]$ only, but analogous results would also hold true with~$c_-[r;b]$.


\subsection{The influence of inverting the direction of the advection}\label{sec21}

We start by addressing our first main question, that is, the effect of the direction of the advection term on the long-time dynamics of~\eqref{eq:KPP}. Namely, for fixed $1$-periodic functions $r\in\mcc^{0,\alpha}(\R)$ and $b\in\mcc^{1,\alpha}(\R)$, we compare $k_0[r;b]$ with $k_0[r;-b]$, and $c_\pm[r;b]$ with $c_\pm[r;-b]$. Although this problem seems natural, it appears that it has not been  investigated until now. We first show that changing $b$ to $-b$ may or may not change the ability of persistence (measured by $k_0[r;\pm b]$) and the spreading speeds. This may occur even if $b$ has zero mean. We next show that, on the contrary, the spreading speeds are not affected by the change of $b$ to $-b$ in the limit of rapidly or slowly oscillating media if $b$ has zero mean. We lastly investigate the effect of the change of $b$ to $-b$ on the ability of persistence in bounded domains with Dirichlet boundary conditions. 

\subsubsection*{General results}

We first show that in general, changing $b$ to $-b$ really has an influence on the ability of persistence and the spreading speeds in general.

\begin{proposition}\label{prop:k0neq}
For every $1$-periodic function $b\in\mcc^{1,\alpha}(\R)$ with $b\not\equiv 0$ and $\ds\int_0^1b=0$, there exists a $1$-periodic function $r\in\mcc^{0,\alpha}(\R)$ such that
$$k_0[r;b]\neq k_0[r;-b],\ \ c_+[r;b]\neq c_+[r;-b],\ \hbox{ and }\ c_+[r;b]\neq c_-[r;-b].$$
\end{proposition}

Let us now focus on the preservation of the ability of persistence and the spreading speeds upon a change of direction of the advection term. Our first result in this direction exhibits two classes of pairs $[r;b]$ such that these quantities are left unchanged.

\begin{theorem}\label{thm:equality}
Let $b\in \mathcal{C}^{1,\alpha}(\R)$ and $r\in\mcc^{0,\alpha}(\R)$ be $1$-periodic functions. Assume that
\begin{enumerate}[label=$(\roman*)$]
\item either there exists $x_0 \in \R$ such that $r(x)=r(x_0-x)$ and $b(x)=b(x_0-x)$ for all $x\in\R$,
\item or there exist $\beta\in\R$ and $\gamma\in\R$ such that $r=\beta+\gamma b$.
\end{enumerate}
Then
$$k_0[r;b]=k_0[r;-b]\ \hbox{ and }\ c_+[r;b]=c_-[r;-b].$$
\end{theorem}

The first item of Theorem~\ref{thm:equality}  is due to a straightforward symmetry: letting $\psi$ be a positive principal eigenfunction of~\eqref{klambda} associated with $\mcl_{\lambda}[r;b]$, the evenness of $r$ and $b$ implies that the function $x\mapsto\psi(-x)$ is a positive eigenfunction associated with $\mcl_{-\lambda}[r;-b]$, which implies that $k_{\lambda}[r;b]=k_{-\lambda}[r;-b]$ for all $\lambda\in\R$, whence $c_+[r;b]=c_-[r;-b]$ by~\eqref{eq:garfre}.

The second item of Theorem~\ref{thm:equality} is unexpected and more involved. It implies in particular that, for constant $r=r_0$ and for any $1$-periodic~$b\in\mcc^{1,\alpha}(\R)$, we have $k_0[r_0;b]=k_0[r_0;-b]$, and $c_+[r_0;b]=c_-[r_0;-b]$. In the proof, under the conditions of item ($ii$), we will exhibit an explicit relationship between the principal eigenfunctions $\phi^+$ and $\phi^-$ of the operators $\mcl_{0}[r;b]$ and $\mcl_{0}[r;-b]$, respectively. For example, when~$b$ has zero mean, this relationship reads:
\[\phi^-(x)=\pth{(\phi^+)'(x)+\gamma\phi^+(x)}\exp\pth{\int_0^xb(y)dy}.\]
The general relationship for $1$-periodic functions $b\in\mcc^{1,\alpha}(\R)$, which may have nonzero mean, is rather involved. Some remarks about the origin of this relationship are given in Section~\ref{s:proof_equality} and in Appendix~\ref{s:construction_phi}. See also Figure~\ref{fig:phi+_phi-} (top panel) for a numerical illustration of the relationship between $\phi^+$ and $\phi^-$ when item~$(ii)$ of Theorem~\ref{thm:equality} is satisfied.

\begin{figure}[h]
\begin{center}
\includegraphics[width=12cm]{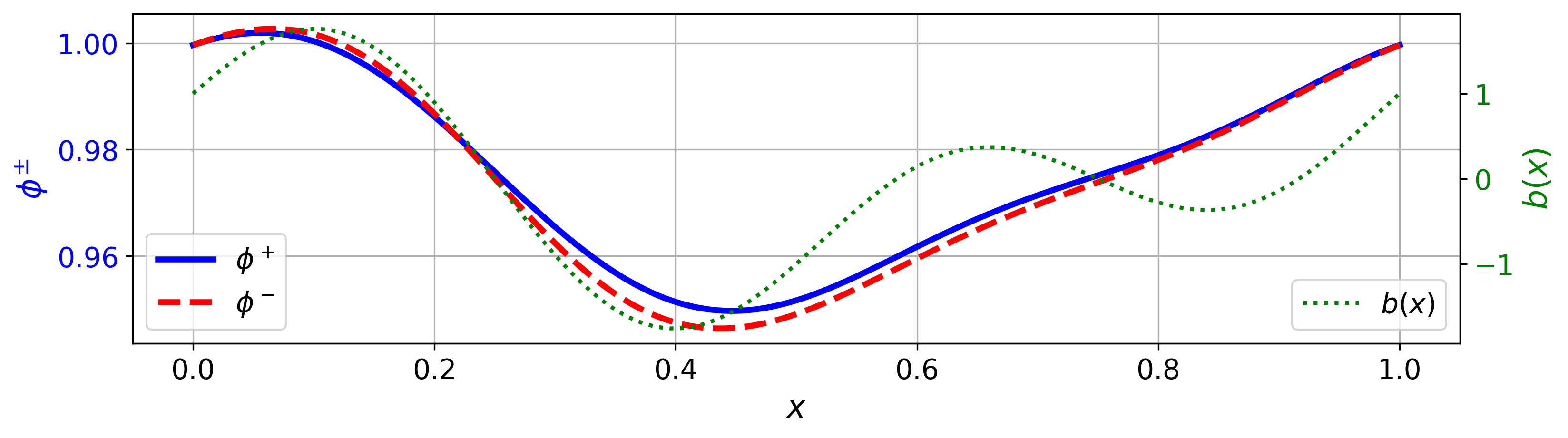}
\includegraphics[width=12cm]{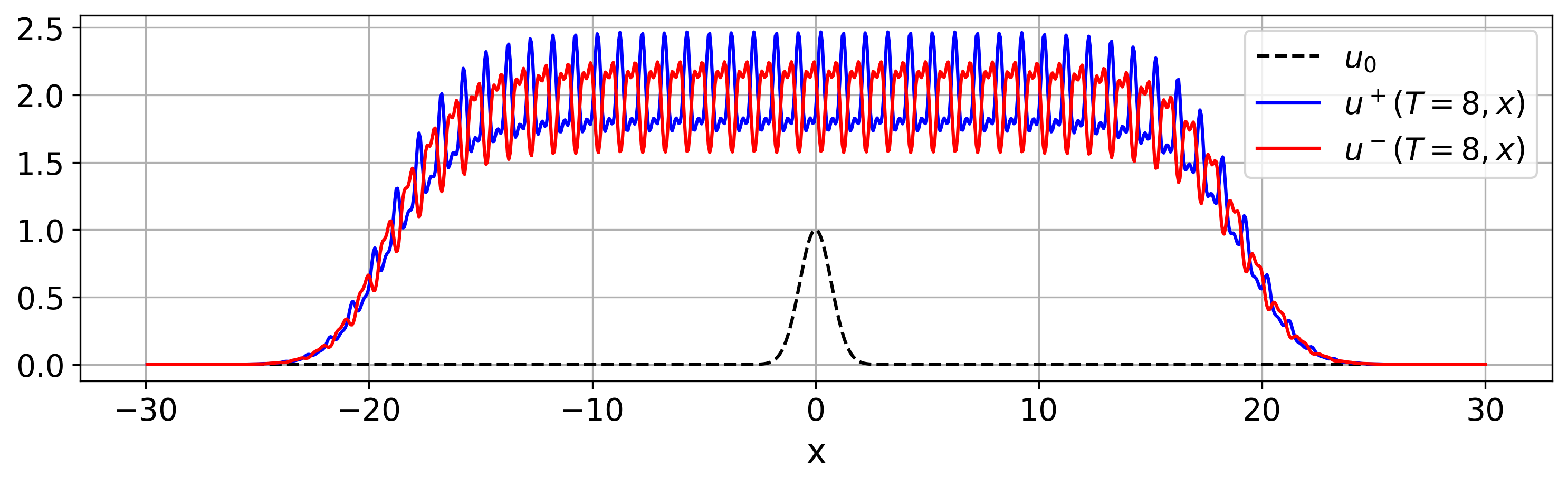}
\includegraphics[width=12cm]{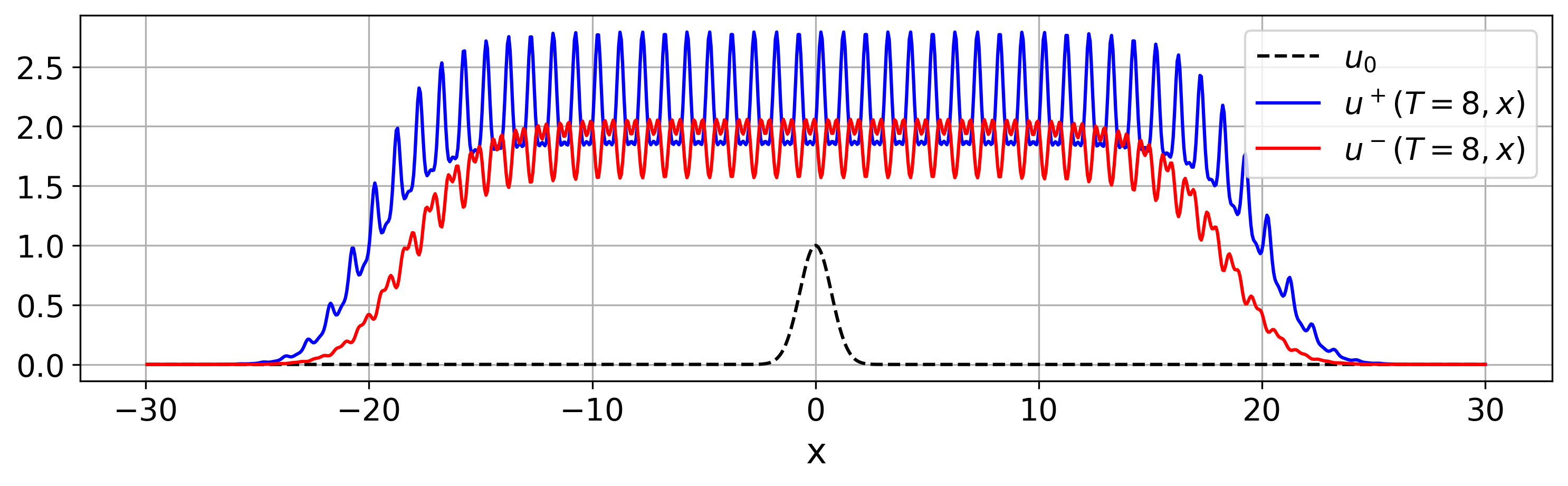}
\end{center}
\caption{\label{fig:phi+_phi-}
The top panel illustrates the nontrivial relationship between $\phi^+$ and $\phi^-$ when item~$(ii)$ of Theorem~\ref{thm:equality} is satisfied: $b(x)=\cos(2\pi x) + \sin(4 \pi x)$ and $r(x)= 2 + b(x)$. The middle panel shows the solutions $u^+ := u[r; b]$ and $u^- := u[r; -b]$ of \eqref{eq:KPP} with the same coefficients as in the top panel. The bottom panel shows the solutions $u^+ := u[r; b]$ and $u^- := u[r; -b]$ of \eqref{eq:KPP} with $b(x) = \cos(2\pi x) + \sin(4 \pi x)$ and $r(x) = 2 + 2 \sin (2 \pi x)$, not satisfying condition~$(ii)$ of Theorem~\ref{thm:equality}.}
\end{figure}

When $b$ has zero mean, the leftward and rightward spreading speeds $c_\pm[r;b]$ coincide, see~\eqref{c-=c+} when $k_0[r;b]>0$, while they are both equal to $-\infty$ when $k_0[r;b]\le0$. Therefore, a consequence of Theorem~\ref{thm:equality} is the following corollary.

\begin{corollary}\label{cor:equality}
Let $b\in \mathcal{C}^{1,\alpha}(\R)$ and $r\in\mcc^{0,\alpha}(\R)$ be $1$-periodic functions. Assume that $\ds\int_0^1b=0$ and that either $(i)$ or $(ii)$ of Theorem~$\ref{thm:equality}$ holds. Then
$$k_0[r;b]=k_0[r;-b]\ \hbox{ and }\ c_+[r; b]=c_+[r; -b]=c_-[r;- b]=c_-[r; b].$$
\end{corollary}

In particular, Theorem~\ref{thm:equality} and Corollary~\ref{cor:equality} imply that, for a constant intrinsic growth term~$r$, the direction of an advection term with zero mean has no influence on the spreading speeds---even this is not at all trivial without any symmetry assumption.

As an illustration of Proposition~\ref{prop:k0neq}, Theorem~\ref{thm:equality}, and Corollary~\ref{cor:equality}, we numerically computed the values of $k_0[r;b]$, $k_0[r;-b]$, and $c_+[r;b]$, $c_+[r;-b]$ for different coefficients~$r$ and $b$, either satisfying or violating condition~$(ii)$ of Theorem~\ref{thm:equality}. Specifically, we took $b(x)=\cos(2\pi x) + \sin(4 \pi x)$. Then, with $r(x)= 2 + b(x)$, we found that $k_0[r;b] = k_0[r;-b] \approx 2.016$ and $c_+[r;b] = c_+[r;-b] \approx 2.819$. In contrast, when  $r(x)= 2  +  2\, \sin(2 \pi x)$, condition~$(ii)$ was not met, yielding $k_0[r;b] \approx 2.221$ and $k_0[r;-b] \approx 1.892$, while $c_+[r;b] \approx 2.933$ and $c_+[r;-b] \approx 2.745$. The corresponding solutions of the Cauchy problem associated with \eqref{eq:KPP}, with a compactly supported initial condition $u_0$, at the fixed time $T=8$, are depicted in Figure~\ref{fig:phi+_phi-} (middle and bottom panels). These numerical computations were performed using the Readi2Solve toolbox~\cite{readi2solve} (see \href{https://readi2solve.biosp.org/}{readi2solve.biosp.org}).

\subsubsection*{Slow or fast oscillations}

To state our second result about the effect of the direction of the advection term on the long-time dynamics of~\eqref{eq:KPP}, we vary the spatial period of the underlying medium and we consider rapidly or slowly oscillating media. To do so, we introduce further notations: for fixed $1$-periodic functions $r\in\mcc^{0,\alpha}(\R)$ and $b\in\mcc^{1,\alpha}(\R)$, we let, for any $L>0$,
\[x\mapsto r_L(x):=r\pth{\frac{x}{L}}\ \hbox{ and }\ x\mapsto b_L(x):=b\pth{\frac{x}{L}}\]
be the $L$-periodic versions of $r$ and $b$. The principal eigenvalues $k_{\lambda}[r_L;b_L]$ are still defined as in~\eqref{klambda}, by the Krein-Rutman theorem, and they are now associated with $L$-periodic principal eigenfunctions. The spreading speeds $c_\pm[r_L;b_L]$ are defined as in~\eqref{eq:garfre} but with $k_{\pm\lambda}[r_L;b_L]$ instead of $k_{\pm\lambda}[r;b]$.

The following result shows in particular that, in rapidly oscillating media ($L\to0$) and in slowly oscillating media ($L\to+\infty$), changing $b_L$ to $-b_L$ does not affect the limiting spreading speeds if $b$ has zero mean.

\begin{proposition}\label{ppn:limits_L}
If $\ds\int_0^1r>0$, then
$$\lim_{L\to0}k_0[r_L;b_L]=\lim_{L\to0}k_0[r_L;-b_L]=\ds\int_0^1r>0$$
and
\begin{equation}\label{limL0}
c_+[r_L;b_L]-c_+[r_L;-b_L]\mathop{\longrightarrow}_{L\to0}\ 2\int_0^1b\quad\text{and}\quad \lim_{L\to0}c_+[r_L;b_L]=\lim_{L\to0}c_-[r_L;-b_L].
\end{equation}
If $\ds\int_0^1b=0$ and $\ds\max_{\R}(r-b^2/4)>0$, then
$$\ds\lim_{L\to+\infty}k_0[r_L;b_L]=\lim_{L\to+\infty}k_0[r_L;-b_L]=\max_{\R}\Big(r-\frac{b^2}{4}\Big)>0$$
and
\[\lim_{L\to+\infty}c_+[r_L;b_L]=\lim_{L\to+\infty}c_+[r_L;-b_L]=\lim_{L\to+\infty}c_-[r_L;b_L]=\lim_{L\to+\infty}c_-[r_L;-b_L].\]
\end{proposition}

The proof of this proposition is based on more or less explicit computations of the limits as $L\to0$ or $L\to+\infty$ of the induced quantities, which are partly provided by~\cite{EHR09,HamFay10,HNR11,Hei10}.

\subsubsection*{Bounded domains and other boundary conditions}

In bounded domains, spreading is no longer possible, but we may still study the ability of persistence. As we only work in the one-dimensional setting, we focus for simplicity on the bounded domain $[-1,1]$.

For $r\in\mcc^{0,\alpha}([-1,1])$ and $b\in\mcc^{1,\alpha}([-1,1])$, we let $k^{d}_{0}[r;b]$ be the principal eigenvalue defined by the Dirichlet eigenproblem
\begin{equation}\label{eq:dir}
\begin{cases}
\phi''+b\phi'+r\phi=k^{d}_{0}[r;b]\phi\ \text{ in $[-1,1]$},\vspace{3pt}\\
\phi(-1)=\phi(1)=0,\ \ \phi>0\hbox{ in $(-1,1)$}.
\end{cases}
\end{equation}

We provide counter-examples to item~$(ii)$ of Theorem~\ref{thm:equality}, showing a strong contrast between the periodic conditions and the Dirichlet 
boundary condition regarding the effect of the direction of the advection on the ability of persistence.

\begin{proposition}\label{ppn:bounded_domain}
Let $b(x):=x$. For any $r\in\mcc^{0,\alpha}([-1,1])$, we have
\[k^{d}_0[r;b]=k^{d}_0[r;-b]-1.\] 
\end{proposition}


\subsection{Relationship between the ability of persistence and the spreading speed}\label{sec22}

We now focus on our second question: \emph{Does a better ability of persistence imply higher spreading speeds of the invasive species?} Although the intuition might suggest that the answer to this question is \emph{yes}, we will construct counter-examples in which increasing the ability of persistence decreases the spreading speeds.

First, we point out that $k_0[r;b]>k_0[\widetilde{r};\widetilde{b}]$ does \emph{not} imply in general that $c_{+}[r;b]>c_{+}[\widetilde{r};\widetilde{b}]$ or similarly $c_{-}[r;b]>c_{-}[\widetilde{r};\widetilde{b}]$. This is a consequence of~\cite[Theorem 3.10]{BKR25}, which implies that, for $r$ with $\max_\R r>0$ and $b$ well-chosen,
\[k_0[r;\gamma b]\mathop{\longrightarrow}_{\gamma\to+\infty}\max_\R r>0\ \hbox{ and }\ c_+[r;\gamma b]\mathop{\longrightarrow}_{\gamma\to+\infty}0.\]
Therefore, with $r_0$ constant such that $0<r_0<\max_\R r$, one has, for all $\gamma>0$ large enough,
\begin{equation}\label{eq:gammalarge}
k_0[r;\gamma b]>r_0=k_0[r_0;0]\ \hbox{ and }\ c_+[r;\gamma b]<2\sqrt{r_0}=c_+[r_0;0].
\end{equation}
Second, the converse implication is also wrong: $c_{+}[r;b]>c_{+}[\widetilde{r};\widetilde{b}]$ (or similarly $c_{-}[r;b]>c_{-}[\widetilde{r};\widetilde{b}]$) does \emph{not} imply in general that $k_0[r;b]>k_0[\widetilde{r};\widetilde{b}]$. Formulas~\eqref{eq:gammalarge} provide some counter-examples. But, even more simply, taking constants $r_0$, $\widetilde{r}_0$, $b_0$ and $\widetilde{b}_0$ with $0<r_0<\widetilde{r}_0$ and $b_0-\widetilde{b}_0>2\sqrt{\widetilde{r}_0}-2\sqrt{r_0}>0$, we have:
$$c_+[r_0;b_0]=2\sqrt{r_0}+b_0>2\sqrt{\widetilde{r}_0}+\widetilde{b}_0=c_+[\widetilde{r}_0;\widetilde{b}_0],$$
but
$$0<k_0[r_0;b_0]=r_0<\widetilde{r}_0=k_0[\widetilde{r}_0,\widetilde{b}_0].$$
Further, note that these examples allow one to make the ratio $|c_{+}[r;b]|/k_0[r;b]$ arbitrarily large or arbitrarily close to~$0$. Therefore, neither $k_0[r;b]$ or $c_{+}[r;b]$ is controlled by the other.

These constructions rely heavily on the fact that we may choose $b\not\equiv0$. The second one makes use of a disequilibrium between spreading to the right and spreading to the left (indeed, we have $c_+[r_0;b_0]>c_+[\widetilde{r}_0,\widetilde{b}_0]$, but $c_-[r_0;b_0]=-b_0+2\sqrt{r_0}<-\widetilde{b}_0+2\sqrt{\widetilde{r}_0}=c_-[\widetilde{r}_0;\widetilde{b}_0]$ since $0<r_0<\widetilde{r}_0$ and $b_0>\widetilde{b}_0$). In fact, even if we require \enquote{equilibrium}, in the sense that 
\[c_{-}[r;b]>c_{-}[\widetilde{r};\widetilde{b}]\ \hbox{ and }\ c_{+}[r;b]>c_{+}[\widetilde{r};\widetilde{b}],\]
we can have $k_0[r;b]<k_0[\widetilde{r};\widetilde{b}]$. To prove it, we will work with $b\equiv0$, in which case the spreading speed to the left and the spreading speed to the right coincide, that is,
$$c_+[r;0]=c_-[r,0]>0,$$
as soon as $k_0[r;0]>0$, as a particular case of~\eqref{c-=c+}.
  
In general, what comparisons can be made when $b\equiv0$? First, it must be recalled that if $r\leq\widetilde{r}$, then $k_0[r;0]\leq k_0[\widetilde{r};0]$; and if further $k_0[r;0]>0$, then $0<c_+[r;0]\leq c_+[\widetilde{r};0]$. This is a consequence of the monotonicity of the eigenvalues $k_\lambda[r;0]=k_{-\lambda}[r;0]$ with respect to $r$. In the sequel, we will go much beyond this immediate observation. On the one hand, we will show that the spreading speed $c_-[r;0]=c_+[r;0]$ is controlled by the ability of persistence $k_0[r;0]$. Namely, if the ability of persistence is low, then the spreading speed must be low as well. On the other hand, we will see that having a low spreading speed does not imply that the ability of persistence is low.
  
\begin{theorem}\label{thm:no_general_comparison}
Let $r\in\mcc^{0,\alpha}(\R)$ be a $1$-periodic function such that $k_0[r;0]>0$.
\begin{enumerate}[label=$(\roman*)$]
\item There exists a $1$-periodic function $\widetilde{r}\in\mcc^{0,\alpha}(\R)$ such that
$$0<k_0[r;0]<k_0[\widetilde{r};0]\ \hbox{ and }\ c_+[r;0]>c_+[\widetilde{r};0]>0.$$
\item More generally speaking, for any $M>0$, there exists a $1$-periodic function $\widetilde{r}\in\mcc^{0,\alpha}(\R)$ such that
\begin{equation}\label{ineq:M}
k_0[\widetilde{r};0]>M\ \hbox{ and }\ 0<c_+[\widetilde{r};0]<\frac1M.
\end{equation}
\item If $r$ is constant, then $0<c_+[\widetilde{r};0]<c_+[r;0]=2\sqrt{r}$ for every non-constant $1$-periodic $\widetilde{r}\in\mcc^{0,\alpha}(\R)$ such that $0<k_0[\widetilde{r};0]\leq k_0[r;0]=r$.\footnote{If $\widetilde{r}$ is a positive constant, then $k_0[\widetilde{r};0]=\widetilde{r}>0$ and $c_+[\widetilde{r};0]=2\sqrt{\widetilde{r}}$, hence the pairs $(k_0[r;0],k_0[\widetilde{r};0])=(r,\widetilde{r})$ and $(c_+[r;0],c_+[\widetilde{r};0])=(2\sqrt{r},2\sqrt{\widetilde{r}})$ are ordered the same way as well.}
\item If $r$ is non-constant, then
\[0<c_+[r;0]<2\sqrt{k_0[r;0]}\]
and there exists a constant $\widetilde{r}>0$ such that
$$0<\widetilde{r}=k_0[\widetilde{r};0]<k_0[r;0]\ \hbox{ and }\ 2\sqrt{\widetilde{r}}=c_+[\widetilde{r};0]>c_+[r;0]>0.$$
\end{enumerate}
\end{theorem}

On the one hand, items~$(i)$ and~$(iv)$ imply that the relationship between the ability of persistence and the spreading speed is not monotonous. In the proof of item~$(i)$, a way to understand this non-monotonicity is to construct an environment with an extremely unfavorable zone, through which individuals will not be able to go, and a very favorable zone, so that the population is yet able to survive. On the other hand, item~$(iii)$ shows that there is still a partial monotonicity if one of the intrinsic growth rates is constant. Lastly, item~$(iv)$ shows that the spreading speed is always controlled by the ability of persistence, while the converse is false, from item~$(ii)$.

\paragraph{Outline of the paper.} Section~\ref{s:proof_equality} is devoted to the proof of Theorem~\ref{thm:equality}, which is our main result about the role of the direction of the advection term. The other results on the role of the direction of the advection term (Propositions~\ref{prop:k0neq},~\ref{ppn:limits_L} and~\ref{ppn:bounded_domain}) are proved in Section~\ref{secadvection}. Section~\ref{s:counter_examples_cvsk} is devoted to the proof of Theorem~\ref{thm:no_general_comparison} on the monotonicity vs. non-monotonicity relationships between the ability of persistence and the spreading speeds in the case $b=0$. Lastly, Appendix~\ref{s:proof_persistence_general_b} deals with a condition for local persistence at large time and Appendix~\ref{s:construction_phi} provides some details on the origin of a crucial change of function used in the proof of Theorem~\ref{thm:equality} in Section~\ref{s:proof_equality}.


\section{A class of pairs $[r;b]$ such that $k_{\lambda}[r;b]=k_{-\lambda}[r;-b]$ and $c_+[r;b]=c_-[r;-b]$: proof of Theorem~\ref{thm:equality}}\label{s:proof_equality}

This section is devoted to the proof of Theorem~\ref{thm:equality}. We first prove item~$(i)$, and then item~$(ii)$.

\begin{proof}[Proof of Theorem~\ref{thm:equality}, item $(i)$] 
Assume that condition~$(i)$ holds. Let $\lambda\in\R$ and let $\psi\in\mcc^{2,\alpha}(\R)$ be a positive, $1$-periodic principal eigenfunction associated with $\mcl_{\lambda}[r;b]$, i.e.
\[\psi''+(b+2\lambda)\psi'+(r+\lambda b+\lambda^2)\psi=k_{\lambda}[r;b]\,\psi.\]
Set $\varphi(x):=\psi(x_0-x)$, for all $x\in\R$. Using $r(x)=r(x_0-x)$ and $b(x)=b(x_0-x)$ for all $x\in\R$, we get:
\[\varphi''+(-b-2\lambda)\varphi'+(r+\lambda b+\lambda^2)\varphi=k_{\lambda}[r;b]\,\varphi.\]
The function $\varphi\in\mcc^{2,\alpha}(\R)$ is $1$-periodic and positive; thus, it is a principal eigenfunction for $\mcl_{-\lambda}[r;-b]$. By uniqueness of the principal eigenvalue, we then have:
\[k_{\lambda}[r;b]=k_{-\lambda}[r;-b].\]
Taking~$\lambda=0$, we obtain $k_{0}[r;b]=k_{0}[r;-b]$. Using the Freidlin-Gärtner formulas~\eqref{eq:garfre}, this also gives $c_+[r;b]=c_-[r;-b]$ if $k_0[r;b]=k_0[r;-b]>0$ (otherwise, if $k_0[r;b]=k_0[r;-b]\le0$, then both $c_+[r;b]$ and $c_-[r;-b]$ are equal to $-\infty$).
\end{proof}

Now, let us prove item~$(ii)$, where~$r$ can be written $r=\beta+\gamma b$. Our goal is to construct a positive $1$-periodic function~$\varphi$ satisfying 
\begin{equation}\label{eq:varphi}
\varphi'' - b \varphi' + (\beta+\gamma b) \varphi = k_0[\beta+\gamma  b ;b]\,\varphi
\end{equation}
in $\R$. Then,~$\varphi$ is a principal eigenfunction of the operator $\mcl_0[\beta+\gamma b;-b]$ associated with the eigenvalue  $k_0[\beta+\gamma  b ;b]$. By uniqueness, this means that $k_0[\beta+\gamma b;-b]=k_0[\beta+\gamma b;b]$. The function $\varphi$ will be obtained as an explicit transformation of a $1$-periodic positive solution $\phi\in\mcc^{2,\alpha}(\R)$ of the principal eigenproblem
\begin{equation}\label{eq:phi+}
\phi''+b\phi'+(\beta+\gamma b)\phi=k_0[\beta+\gamma b;b]\,\phi
\end{equation}
in $\R$. To construct the function $\varphi$, we first prove the following property of any solution $\phi\in\mcc^2(\R)$ of~\eqref{eq:phi+}, even non-periodic.

\begin{lemma}\label{lem:sign_phi'+phi}
Let $b$ be a continuous function and let $\gamma\in\R$. Let $\phi\in\mcc^2(\R)$ be positive in $\R$ $($we do not require $b$ or $\phi$ to be periodic$)$ and let $k\in\R$ satisfy
\begin{equation}\label{eq:phi}
\mathcal{L}_{0}[\gamma\,b;b]=k\phi
\end{equation}
in $\R$. Then, either $\phi'+\gamma\phi\equiv0$ in $\R$, or $\phi' + \gamma \phi$ vanishes at most once in $\R$.
\end{lemma}

\begin{proof}
We consider three separate cases depending on the value of $k$. Consider first the case $k>\gamma^2$. Assume that there exists $x_0\in\R$ such that $\phi'(x_0)+\gamma \phi(x_0)=0$. We have $ \phi''(x_0) + b(x_0) (\phi'(x_0)+\gamma\phi(x_0))=k\phi(x_0)$, which implies that $\phi''(x_0)  = k \phi(x_0)$. Thus, recalling that $\phi(x_0)>0$ and $k>\gamma^2$,
$$(\phi'+\gamma \phi)'(x_0)=\phi''(x_0)+\gamma \phi'(x_0)=k\phi(x_0)+\gamma \phi'(x_0)>\gamma \left( \gamma\phi(x_0)+\phi'(x_0)\right)=0.$$
Therefore, at each point where the function $\phi' + \gamma \phi$ vanishes, its derivative $(\phi'+\gamma \phi)'$ is positive. As a consequence, $\phi' + \gamma \phi$ can vanish at most once in $\R$.

The same arguments as in the previous paragraph imply that, if $k<\gamma^2$ and if $\phi'(x_0)+\gamma\phi(x_0)=0$ for some $x_0\in\R$, then $(\phi'+\gamma \phi)'(x_0)<0$. This again shows that $\phi'+\gamma\phi$ vanishes at most once in $\R$.

Lastly, consider the case $k=\gamma^2$. Dividing  \eqref{eq:phi} by the positive function $\phi$, we obtain:
$$\frac{\phi''}{\phi} + b \frac{\phi'}{\phi} + \gamma b = \gamma^2.$$
Setting $w:=\frac{\phi'}{\phi}+\gamma$, it follows that $(w-\gamma)'+ (w-\gamma)^2 + b (w-\gamma) + \gamma b = \gamma^2$ in $\R$, thus $w' = w (2\gamma - b - w)$ in $\R$. By the Cauchy-Lipschitz theorem, either $w\equiv 0$ in $\R$, or $w$ does not vanish in $\R$. The proof of Lemma~\ref{lem:sign_phi'+phi} is thereby complete.
\end{proof}

We now turn to the proof of Theorem~\ref{thm:equality} with assumption~$(ii)$.
We first prove the result for the ability of persistence, and then for the spreading speed.

\begin{proof}[Proof of Theorem~\ref{thm:equality}, item~$(ii)$, for $k_0$]
Assume that there exist $\beta\in\R$ and $\gamma\in\R$ such that $r=\beta+\gamma b$ in $\R$, and let us prove that $k_0[r;b]=k_0[r;-b]$. Without loss of generality, we may assume that $\beta=0$, since $k_0[\beta+\gamma b;\pm b]=\beta+k_0[\gamma b;\pm b]$. Moreover, if $r=\gamma b$ with $\gamma=0$, then $k_0[r;b]=k_0[0;b]=0=k_0[0;-b]=k_0[r;-b]$, so the result holds.

From now on, we assume that $\beta=0$ and $\gamma\neq0$, that is, $r=\gamma b$ with $\gamma\neq0$. Remember that $\phi\in\mcc^{2,\alpha}(\R)$ is a $1$-periodic positive principal eigenfunction of~\eqref{eq:phi+}. The positivity and periodicity of $\phi$ and the fact that $\gamma\neq0$ imply that $\phi'+\gamma\phi\not\equiv0$ in $\R$. Therefore, using Lemma~\ref{lem:sign_phi'+phi}, the function $\phi'+\gamma\phi$ vanishes at most once in $\R$. Hence, by periodicity, the function $\phi'+\gamma\phi$ does not vanish in $\R$, and then has a constant strict sign in $\R$. Thus, the $\mcc^2(\R)$ function $g$ given in $\R$ by
$$g(x) := \int_0^x \frac{\gamma\,\phi'(y) + k_0[\gamma \, b ;b] \phi(y)}{\phi'(y) + \gamma\phi(y)} \, dy$$
is well-defined. Now, set, for $x\in\R$,
$$h(x):=\int_0^x\frac{e^{-g(y)}}{\phi'(y)+\gamma\phi(y)}\,dy.$$
The function $h$ is of class $\mcc^2(\R)$ and, since $h'$ has a constant strict sign in $\R$ and $h(0)=0$, we have $h(1)\neq 0$. We then set, for $x\in\R$,
\begin{equation}\label{defvarphi}
\varphi(x) := e^{g(x)}(1+Ch(x)),\ \hbox{ with }C := \frac{e^{-g(1)} - 1}{h(1)},
\end{equation}
whence $\varphi(0)=1=\varphi(1)$. The function $\varphi$ defined is of class $\mcc^2(\R)$. Furthermore, $\varphi$ is positive in $[0,1]$ since the function $1+Ch$ is monotone in $\R$ and takes positive values, $1$ and $e^{-g(1)}$, at $0$ and $1$, respectively. We refer to Appendix~\ref{s:construction_phi} for the motivation for this expression of~$\varphi$.

For convenience, we now define the differential operator
$$\T : \mu \mapsto \mu''-b\mu'+(\gamma b-k_0[\gamma\,b;b])\,\mu$$
acting on $\mcc^2(\R)$ functions $\mu$. Let us first show that $\T [e^g]=0$ in $\R$. We have 
$$\T[e^g]=e^g\big(g''+(g')^2 - b g' + \gamma b - k_0[\gamma \, b ;b]\big)$$
in $\R$. Using Equation~\eqref{eq:phi+}, satisfied by $\phi$, there holds
$$b=\frac{k_0[\gamma \, b ;b] \phi -\phi''}{\phi'+\gamma\phi}$$
in $\R$. Together with the formulas
\begin{equation}\label{eq:g'}
g'=\frac{\gamma\phi' + k_0[\gamma \, b ;b] \phi}{\phi' + \gamma \phi}\ \hbox{ and }\ g''=(\gamma^2-k_0[\gamma \, b ;b])\frac{\phi''\phi-(\phi')^2}{(\phi'+\gamma\phi)^2}
\end{equation}
in $\R$, we obtain that $\T[e^g]=0$ in $\R$.

Let us then show that $\T [\varphi]=0$ in $\R$. We have 
$$\T [\varphi]=\T [e^g(1+C \,h)]=\T [e^g]+ C\,\T[e^g h]=C\,\T[e^g h]$$
since $\T[e^g]=0$ in $\R$, whence
\begin{align*}
\T[\varphi] & = C\big([e^g h]''- b [e^g h]' +(\gamma b-k_0[\gamma \, b ;b])e^g h\big) \\
& = C\,e^g\left[(g''+(g')^2 - b g' + \gamma b - k_0[\gamma \, b ;b])h + h''+  2 h' g' - b h'\right] \\
& = C\,\T [e^g]h + C\,e^g\left[h''+  2 h' g' - b h'\right] \\
& = C\,e^g\left[h''+  2 h' g' - b h'\right],
\end{align*}
again because $\T[e^g]=0$ in $\R$. Since
\[h'=\frac{e^{-g}}{\phi'+\gamma\phi}\ \hbox{ and }\ h''=e^{-g}\left[\frac{-g'}{\phi'+\gamma\phi}-\frac{\phi''+\gamma\phi'}{(\phi'+\gamma\phi)^2}\right]\]
in $\R$, together with~\eqref{eq:phi+} and~\eqref{eq:g'}, we get $h''+  2 h' g' - b h' = 0$ in $\R$, which entails that
$$\T[\varphi]=0\ \hbox{ in $\R$}.$$

Let us finally show that the function $\varphi$ is $1$-periodic. Recalling that $\varphi(0)=\varphi(1)=1$ owing to the definition of $C$, it is thus sufficient to show that $\varphi'(0)=\varphi'(1)$, by the Cauchy-Lipschitz theorem. On the one hand, 
$$\varphi'=e^g\left[g'(1+C h)+ C h'\right].$$
On the other hand, since $\phi$ is $1$-periodic, we have 
$$g'(0)=g'(1)=\frac{\gamma\phi'(0)+k_0[\gamma \, b ;b]\phi(0)}{\phi'(0)+\gamma \phi(0)}$$
and $h'(1)=e^{-g(1)}h'(0)$. Thus, owing again to the definition of $C$,
\begin{align*}
\varphi'(1) & = e^{g(1)}\left[g'(1)(1+C h(1))+ C h'(1)\right] \\
& = g'(0) \, e^{g(1)} (1+C h(1)) +  C h'(0)= g'(0) + C h'(0) = \varphi'(0).
\end{align*}
Therefore,~$\varphi$ is $1$-periodic and, since it is positive in $[0,1]$, it is positive in $\R$.

To conclude, $(\varphi, k_0[\gamma b ;b])$  forms a pair consisting of the principal eigenfunction and the principal eigenvalue of the operator $\mathcal{L}_0[ \gamma b ;-b]: \psi \mapsto \psi'' - b \psi' + \gamma b\psi$, with $1$-periodic conditions. By uniqueness of the principal eigenvalue, we deduce that $k_0[\gamma b ;-b] = k_0[\gamma b ;b]$. Furthermore, by uniqueness of the principal eigenfunction of $\mathcal{L}_0[\gamma b;-b]$ up to multiplication, it follows that $\varphi$ is equal to the unique principal eigenfunction of $\mathcal{L}_0[\gamma b;-b]$ satisfying the normalization $\varphi(0)=1$.
\end{proof}

\begin{proof}[Proof of  Theorem~\ref{thm:equality}, item $(ii)$, for the speeds.]
Let $\lambda\in\R$, and let us show that
$$k_{\lambda}[\beta+\gamma b;b]=k_{-\lambda}[\beta+\gamma b;-b].$$
The eigenvalue $k_{\lambda}[\beta+\gamma b;b]$ is associated with the operator
\begin{align*}
\psi\mapsto\mcl_{\lambda}[\beta+\gamma b;b]\psi&=\psi''+(b+2\lambda)\psi'+(\beta+\gamma b+\lambda b+\lambda^2)\psi\\
&=\psi''+(b+2\lambda)\psi'+(\gamma+\lambda)(b+2\lambda)\psi+\beta\psi-(\lambda^2+2\lambda\gamma)\psi
\end{align*}
acting on $1$-periodic functions of class $\mcc^2(\R)$. Therefore, we have
\[k_{\lambda}[\beta+\gamma b;b]=k_{0}[(\gamma+\lambda) (b+2\lambda);b+2\lambda]+\beta-(\lambda^2+2\lambda\gamma).\]
We have already proved that if $\widetilde{r}=\widetilde{\gamma}\widetilde{b}$ for some $\widetilde{\gamma}\in\R$, then $k_0[\widetilde{\gamma}\widetilde{b};\widetilde{b}]=k_0[\widetilde{\gamma}\widetilde{b};-\widetilde{b}]$. Taking $\widetilde{\gamma}=\gamma+\lambda$ and $\widetilde{b}=b+2\lambda$, we get
\[k_{\lambda}[\beta+\gamma b;b]=k_{0}[(\gamma+\lambda) (b+2\lambda);-(b+2\lambda)]+\beta-(\lambda^2+2\lambda\gamma).\]
Hence, $k_{\lambda}[\beta+\gamma b;b]$ is also the principal eigenvalue associated with the operator
\begin{align*}
\psi\ \mapsto
&\ \ \ \ \ \ \psi''-(b+2\lambda)\psi'+(\gamma+\lambda)(b+2\lambda)\psi+\beta\psi-(\lambda^2+2\lambda\gamma)\psi\\
&\ \ =\psi''-(b+2\lambda)\psi'+(\beta+\gamma b+\lambda b+\lambda^2)\psi\\
& \ \ =\mcl_{-\lambda}[\beta+\gamma b;-b]\psi.
\end{align*}
As a consequence, $k_{\lambda}[\beta+\gamma b;b]=k_{-\lambda}[\beta+\gamma b;-b]$ for all $\lambda\in\R$. This implies, using the Freidlin-Gärtner formulas~\eqref{eq:garfre}, that $c_+[\beta+\gamma b;b]=c_-[\beta+\gamma b;-b]$. The proof of Theorem~\ref{thm:equality} is thereby complete.
\end{proof}


\section{Other results on the influence of the direction of the advection term}\label{secadvection}

This section is devoted to the proofs of Propositions~\ref{prop:k0neq},~\ref{ppn:limits_L} and~\ref{ppn:bounded_domain}.


\subsection{General effect of the direction of the advection term: proof of Proposition~\ref{prop:k0neq}}

\begin{proof}[Proof of Proposition \ref{prop:k0neq}]
Let $b\in\mcc^{1,\alpha}(\R)$ be a $1$-periodic function with $b\not\equiv0$ and $\ds\int_0^1b=0$. Let $\lambda\in\R$ and let $\phi>0$ be a $\mcc^{2,\alpha}(\R)$ principal eigenfunction of $\mcl_{\lambda}[r;b]$, where $r\in\mcc^{1,\alpha}(\R)$ is a $1$-periodic function that will be chosen later. Consider the transformation $\phi=F e^{-B/2}$, with $B'=b$. The function $F\in\mcc^2(\R)$ satisfies
\[F''+2\lambda F'+\pth{r-\frac{b^2}{4}-\frac{b'}{2}+\lambda^2}F=k_{\lambda}[r;b]\,F\]
in $\R$. Furthermore, since $\phi$ is $1$-periodic and positive and since~$b$ is $1$-periodic with zero mean, the function~$F$ is also $1$-periodic and positive. The function~$F$ is therefore a principal eigenfunction of the operator $\mathcal{L}_{\lambda}\big[r-\frac{b^2}{4}-\frac{b'}{2};0\big]$. Thus, by uniqueness of the principal eigenvalue, we obtain that (see also~\cite{Nad10})
\[k_{\lambda}[r,b]=k_{\lambda}\cro{r-\frac{b^2}{4}-\frac{b'}{2};0}.\]
Similarly, we have
\[k_{\lambda}[r,-b]=k_{\lambda}\cro{r-\frac{b^2}{4}+\frac{b'}{2}; 0}.\]
Consider now any positive real number $\eta>0$ and set
$$r:=\frac{b^2}{4}+\frac{b'}{2}+\eta,$$
which is $1$-periodic and of class $\mcc^{0,\alpha}(\R)$. Then $k_{\lambda}[r;b]=k_{\lambda}[\eta;0]=\lambda^2+\eta$. Furthermore, since $b\not\equiv0$ and $b$ has zero mean, the function $b'\in\mcc^{0,\alpha}(\R)$ is not constant, so the principal $1$-periodic eigenfunction $\psi$ of the operator $\mathcal{L}_\lambda[b'+\eta;0]$ (which satisfies
\[\psi''+2\lambda\psi'+(b'+\lambda^2+\eta)\psi=k_\lambda[b'+\eta;0]\,\psi\]
in $\R$) is not constant. By integrating this equation against $1/\psi$ over $[0,1]$, one gets
$$k_\lambda[b'+\eta;0]=\int_0^1\frac{(\psi')^2}{\psi^2}+\lambda^2+\eta>\lambda^2+\eta.$$
Thereby, $k_{\lambda}[r;-b]=k_\lambda[b'+\eta;0]> \lambda^2+\eta=k_{\lambda}[r;b]$. Since this holds for all $\lambda\in\R$, we get in particular that $0<\eta=k_{0}[r;b]<k_{0}[r;-b]$. Moreover, since each infimum in~\eqref{eq:garfre} is reached at some $\lambda>0$, we deduce that $c_+[r;b]<c_+[r;-b]$. Last, since $b$ has zero mean, we have $c_+[r;-b]=c_-[r;-b]$ by~\cite[Proposition~3.9]{BKR25}, so $c_+[r;b]<c_-[r;-b]$.
\end{proof}


\subsection{Small and large periods: proof of Proposition~\ref{ppn:limits_L}}

We use the results of~\cite{EHR09,HNR11} to compute more or less explicitly the limits as $L\to0$ and $L\to+\infty$ of $k_0[r_L;b_L]$ and $c_+[r_L;b_L]$. Before that, we extend some results of~\cite{EHR09} to the case where there is a nonzero advection term.

\begin{lemma}\label{lem:limitsLto0}
Let $r\in\mcc^{0,\alpha}(\R)$ and $b\in\mcc^{1,\alpha}(\R)$ be $1$-periodic functions, and note
\[\overline{r}:=\int_0^1r(x)dx\ \hbox{ and }\ \ \overline{b}:=\int_0^1b(x)dx.\]
Then, for every $\lambda\in\R$,
\begin{equation}\label{eq:convk}
k_{\lambda}[r_L;b_L]\mathop{\longrightarrow}_{L\to0}\ \overline{r}+\lambda\overline{b}+\lambda^2.
\end{equation}
If $\overline{r}>0$, then
$$c_{\pm}[r_L;b_L]\mathop{\longrightarrow}_{L\to0}\ \pm\overline{b}+2\sqrt{\overline{r}}.$$
\end{lemma}

\begin{proof}
We first show that for every $\lambda\in\R$, $k_{\lambda}[r_L;b_L]\to \overline{r}+\lambda \overline{b}+\lambda^2$ as $L\to 0$. For $L>0$, let $\phi_L\in\mcc^{2,\alpha}(\R)$ be the positive $L$-periodic principal eigenfunction of the adjoint of $\mcl_{\lambda}[r_L;b_L]$, that is,
\begin{equation}\label{eq:phi_L_L0}
  \phi_L''-\cro{(b_L+2\lambda)\phi_L}'+(r_L+\lambda b_L+\lambda^2)\phi_L=k_{\lambda}[r_L;b_L]\,\phi_L
\end{equation}
in $\R$, with normalization
\begin{equation}\label{eq:phi_normalisation}
  \int_0^1\phi_L^2(x)dx=1.
\end{equation}
By multiplying~\eqref{eq:phi_L_L0} by $\phi_L$ and integrating by parts over $[0,1]$, we get:
\begin{equation}\label{eq:intphi}
-\int_0^1(\phi_L')^2+\int_0^1(b_L+2\lambda)\phi_L\phi_L'+\int_0^1(r_L+\lambda b_L+\lambda^2)\phi_L^2=k_{\lambda}[r_L;b_L]\int_0^1\phi_L^2.
\end{equation}
On the other hand, using~\eqref{kbornes} with $(r_L,b_L)$ instead of $(r,b)$, it follows that the family $(k_{\lambda}[r_L;b_L])_{L>0}$ is bounded. Moreover, by Young's inequality, we have
$$|(b_L+2\lambda)\phi_L\phi_L'|\le\frac{(\phi_L')^2}{2}+\frac{(b_L+2\lambda)^2\phi_L^2}{2}.$$
Hence, together with~\eqref{eq:phi_normalisation}-\eqref{eq:intphi}, we conclude that the family $(\phi_L)_{L>0}$ is bounded in $H^1([0,1])$. Therefore, there exist a sequence $(L_n)_{n\in\N}$ of positive real numbers converging to $0$ and a function $\phi\in H^1([0,1])$ such that $\phi_{L_n}\to\phi$ in $\mcc^{0,\beta}([0,1])$ for all $\beta\in[0,1/2)$ (by Morrey's inequality), and $\phi_{L_n}\toweak\phi$ in $H^1([0,1])$ weakly. Last, since each function $\phi_{L_n}$ is $L_n$-periodic, we infer that $\phi\equiv C$ is a non-negative constant with $C^2=1$, that is, $\phi\equiv1$ in $[0,1]$.

Now, multiplying~\eqref{eq:phi_L_L0} (with $L=L_n$) by any function $v\in\mcc^1_c((0,1))$ and integrating by parts over $[0,1]$, we get:
\begin{equation}\label{eq:weak_phi_L}
-\int_0^1\phi_{L_n}'v'+\int_0^1(b_{L_n}+2\lambda)\phi_{L_n}v'+\int_0^1(r_{L_n}+\lambda b_{L_n}+\lambda^2)\phi_{L_n}v=k_{\lambda}[r_{L_n};b_{L_n}]\int_0^1\phi_{L_n}v.
\end{equation}
The sequence $(k_{\lambda}[r_{L_n};b_{L_n}])_{n\in\N}$ is bounded, so there exists $\kappa\in\R$ such that, up to extraction of a subsequence, $k_{\lambda}[r_{L_n};b_{L_n}]\to\kappa$ as $n\to+\infty$. As $n\to+\infty$, we have $r_{L_n}\toweak\overline{r}$ weakly in $L^2([0,1])$ and $b_{L_n}\toweak\overline{b}$ weakly in $L^2([0,1])$. Hence, taking the limit $n\to+\infty$ in~\eqref{eq:weak_phi_L} and using the fact that the sequence $(\phi_{L_n})_{n\in\N}$ converges to $1$ strongly in (at least) $L^2([0,1])$ and weakly in $H^1([0,1])$, we infer that
\begin{equation*}
\kappa\int_0^1v=\int_0^1(\overline{b}+2\lambda)v'+\int_0^1(\overline{r}+\lambda \overline{b}+\lambda^2)v=(\overline{r}+\lambda \overline{b}+\lambda^2)\int_0^1v.
\end{equation*}
Since this holds for every $v\in\mcc^1_c((0,1))$, we obtain that $\kappa=\overline{r}+\lambda \overline{b}+\lambda^2$. This value~$\kappa$ does not depend on any sequence $(L_n)_{n\in\N}$ converging to $0$, so we deduce that the whole bounded family $(k_{\lambda}[r_L;b_L])_{L>0}$ converges to $\overline{r}+\lambda \overline{b}+\lambda^2$ as $L\to0$, that is,~\eqref{eq:convk} is proved.

Let us now show the convergence of $c_+[r_L;b_L]$ as $L\to0$, assuming that $\overline{r}>0$, that is, $\lim_{L\to0}k_0[r_L;b_L]>0$. For each $\lambda\in\R$, the function $L\mapsto k_{\lambda}[r_L;b_L]$ is non-decreasing, by~\cite[Proposition 4.1]{Nad10}.\footnote{The part of the proof of~\cite[Proposition 4.1]{Nad10} related to the monotonicity of $k_{\lambda}[r_L;b_L]$ with respect to $L>0$ actually holds for any $1$-periodic $b\in\mcc^{1,\alpha}(\R)$, even not divergence-free, that is, here, not constant.} Therefore, $k_0[r_L;b_L]>0$ for all $L>0$, and the functions $L\mapsto c_\pm[r_L;b_L]$ are non-decreasing. Hence
$$\lim_{L\to 0}c_{\pm}[r_L;b_L]=\inf_{L> 0}c_{\pm}[r_L;b_L]=\inf_{L>0,\, \lambda>0}\frac{k_{\pm\lambda}[r_L;b_L]}{\lambda}=\inf_{\lambda>0}\Big(\lim_{L\to 0}\frac{k_{\pm\lambda}[r_L;b_L]}{\lambda}\Big).$$
From~\eqref{eq:convk}, we conclude that
$$\lim_{L\to 0}c_{\pm}[r_L;b_L]=\inf_{\lambda>0}\Big(\frac{\overline{r}}{\lambda}\pm\overline{b}+\lambda\Big)=\pm\overline{b}+2\sqrt{\overline{r}},$$
as stated. 
\end{proof}

\begin{proof}[Proof of Proposition~\ref{ppn:limits_L}, for $L\to0$]
Using Lemma~\ref{lem:limitsLto0}, we note that the limit as $L\to0$ of $k_0[r_L;\pm b_L]$ is $\overline{r}$ and does not depend on whether we take $+b_L$ or $-b_L$. Replacing $b$ by $-b$ when needed in Lemma~\ref{lem:limitsLto0}, we also obtain the two limits~\eqref{limL0}.
\end{proof}

\begin{proof}[Proof of Proposition~\ref{ppn:limits_L}, for $L\to+\infty$]
Assume that $b$ has zero mean, that is, $\overline{b}=0$. First, using~\cite[Proposition~3.2]{HNR11}, we have 
\[k_0[r_L;\pm b_L]\mathop{\longrightarrow}_{L\to+\infty}\ M:=\max_{x\in\R}\pth{r(x)-\frac{b(x)^2}{4}}.\]
Hence $k_0[r_L; b_L]$ and $k_0[r_L;-b_L]$ converge to the same limit. Second, assuming $M>0$, both quantities $k_0[r_L; b_L]$ and $k_0[r_L;-b_L]$ are positive for all $L$ large enough, and using~\cite[Theorem~2.3]{HNR11} we have
\[c_+[r_L;\pm b_L]\mathop{\longrightarrow}_{L\to+\infty}\ \min_{\lambda\ge j(M),\,\lambda>0}\frac{j^{-1}(\lambda)}{\lambda},\]
where $j:[M,+\infty)\to[j(M),+\infty)$ is the bijective function defined by
\[j(k):=\int_0^1\sqrt{k-r(x)+\frac{b(x)^2}{4}}dx.\]
Hence, $\lim_{L\to+\infty}c_+[r_L; b_L]=\lim_{L\to+\infty}c_+[r_L;- b_L]$. Last, by~\eqref{c-=c+}, $c_+[r_L;\pm b_L]=c_-[r_L;\pm b_L]$ for all $L$ large enough. The conclusion of Proposition~\ref{ppn:limits_L} follows.
\end{proof}


\subsection{A counter-example to Theorem~\ref{thm:equality} in bounded domains: proof of Proposition~\ref{ppn:bounded_domain}}

\begin{proof}[Proof of Proposition~\ref{ppn:bounded_domain}]
  Let $\mcl_{0}[r;b]$ be the usual operator, but with Dirichlet
  boundary conditions, see~\eqref{eq:dir}. Since $b(x)=x$, we have
$$\mcl_0[r;b]\phi=\phi''+x\phi'+r\phi=\phi''-(-x\phi)'+(r-1)\phi$$
  for any $\phi\in\mcc^2([-1,1])$. Using the Dirichlet boundary conditions, 
  we deduce that:
$$\mcl_0[r;b]=\mcl_0^*[r;-b]-1.$$
  Since the operators $\mcl_0[r;-b]$ and $\mcl_0^*[r;-b]$ (both with Dirichlet boundary conditions) 
  have the same principal eigenvalue, we conclude that $k_0^{d}[r;b]=k_0^{d}[r;-b]-1$. 
\end{proof}


\section{Relationship between the spreading speed and the ability of persistence: proof of Theorem~\ref{thm:no_general_comparison}}\label{s:counter_examples_cvsk}

This section is devoted to the proof of Theorem~\ref{thm:no_general_comparison}. Throughout this section, we assume that $b=0$ and, as shortcuts, we then write, for any $1$-periodic function $r\in\mcc^{0,\alpha}(\R)$,
$$k_\lambda[r]:=k_\lambda[r;0],\ \ \mathcal{L}_\lambda[r]:=\mathcal{L}_\lambda[r;0],\ \hbox{ and }\ c_\pm[r]:=c_\pm[r;0].$$
Hence, if $k_0[r]>0$, then $c_-[r]=c_+[r]>0$, by~\eqref{c-=c+}.

\begin{proof}[Proof of Theorem~\ref{thm:no_general_comparison}, items~$(i)$ and $(ii)$]
Since $(ii)$ is stronger than $(i)$, we directly show $(ii)$. Let $M>0$ be any fixed positive real number. Define
$$R:=M+16\pi^2>0.$$
For any $A\in\R$, we then consider the Lipschitz-continuous $1$-periodic function $\widetilde{r}^A$ such that
$$\widetilde{r}^A(x):=\left\{\begin{array}{ll}
R & \text{for $0\leq x\le1/4$},\vspace{3pt}\\
R+4(A-R)(x-1/4) & \text{for $1/4<x<1/2$},\vspace{3pt}\\
A & \text{for $1/2\leq x\le3/4$},\vspace{3pt}\\
R+4(A-R)(1-x) & \text{for $3/4<x<1$}\end{array}\right.$$
(this function is then extended by $1$-periodicity in $\R$). For each $\lambda\in\R$, let $k_{\lambda,\, dir}$ be the principal Dirichlet eigenvalue of
\begin{equation}\label{eq:eig_pb_dir}
\left\{\begin{array}{ll}
\psi''+2\lambda\psi'+(R+\lambda^2)\psi=k_{\lambda}^{d}\psi & \text{in $[0,1/4]$},\vspace{3pt}\\
\psi>0 & \text{in $(0,1/4)$},\vspace{3pt}\\
\psi(0)=\psi(1/4)=0,&
\end{array}
\right.
\end{equation}
associated with the $\mcc^\infty([0,1/4])$ principal eigenfunction $\psi:x\mapsto e^{-\lambda x}\sin(4\pi x)$. By uniqueness of the principal eigenvalue, it follows that $k_{\lambda}^{d}=R-16\pi^2$. On the other hand, from~\cite[Equation 1.13]{BNV94}, 
\begin{equation}\label{klambdadir}
k_{\lambda}^{d}=\min_{\varphi\in W^{2,1}_{loc}((0,1/4)),\,\varphi>0\hbox{\small{ in }}(0,1/4)}\pth{\mathop{\mathrm{esssup}}_{x\in(0,1/4)}\frac{\varphi''(x)+2\lambda\varphi'(x)+(R+\lambda^2)\varphi(x)}{\varphi(x)}}
\end{equation}
and the minimum with respect to $\varphi$ is reached only by the positive multiples of $\psi$. Hence $k_{\lambda}^{d}<k_\lambda[\widetilde{r}^A]$ for each $A\in\R$, since $\widetilde{r}^A=R$ in $[0,1/4]$ and $k_\lambda[\widetilde{r}^A]$ is associated with a positive $1$-periodic principal eigenfunction of $\mathcal{L}_\lambda[\widetilde{r}^A]$, whose restriction to $[0,1/4]$ is therefore still positive in $[0,1/4]$. In other words, for every $A\in\R$,
$$k_\lambda[\widetilde{r}^A]>R-16\pi^2=M.$$
In particular, $k_0[\widetilde{r}^A]>M>0$, which is one the two required conclusions of~\eqref{ineq:M}.

Our goal is now to show that $c_-[\widetilde{r}^A]=c_+[\widetilde{r}^A]<1/M$ for all $A<0$ negative enough. To do so, for any fixed $\lambda\in\R$, let us first compute the limit of $k_\lambda[\widetilde{r}^A]$ as $A\to-\infty$. If $A_1\le A_2$, then $\widetilde{r}^{A_1}\le\widetilde{r}^{A_2}$ in $\R$, whence $k_\lambda[\widetilde{r}^{A_1}]\le k_\lambda[\widetilde{r}^{A_2}]$. Therefore, there exists a limit
$$\ell:=\lim_{A\to-\infty}k_\lambda[\widetilde{r}^A],$$
such that $k_\lambda[\widetilde{r}^0]\ge\ell\ge R-16\pi^2=M$. We show in what follows that $\ell=R-16\pi^2$ (as a consequence,~$\ell$ is independent of $\lambda$). For $A\le0$, let $\phi_A>0$ be $1$-periodic of class $\mcc^2(\R)$ and satisfy
\begin{equation}\label{eq:eigv_pb_phiA}
\phi''_A+2\lambda\phi_A'+(\widetilde{r}^A+\lambda^2)\phi_A=k_\lambda[\widetilde{r}^A]\phi_A
\end{equation}
with normalization $\|\phi_A\|_{L^2([0,1])}=1$. Integrating~\eqref{eq:eigv_pb_phiA} against $\phi_A$ over $[0,1]$, we get:
\[k_\lambda[\widetilde{r}^A]=-\int_0^1(\phi'_A)^2+\lambda\underbrace{\int_0^1(\phi^2_A)'}_{=0}+\lambda^2+\int_0^1\underbrace{\widetilde{r}^A}_{\le R}\phi_A^2\le-\int_0^1(\phi'_A)^2+\lambda^2+R.\]
Since the family $(k_\lambda[\widetilde{r}^A])_{A\le0}$ is bounded (it ranges in $[M,k_\lambda[\widetilde{r}^0]]$) and since $\|\phi_A\|_{L^2([0,1])}=1$, we infer that $(\phi_A)_{A\le0}$ is bounded in $H^1([0,1])$. Therefore, there exist a sequence $(A_n)_{n\in\N}$ of non-positive real numbers diverging to $-\infty$ and a non-negative $1$-periodic function $\phi\in H^1_{loc}(\R)$ such that $\phi_{A_n}\to\phi$ in $\mcc^{0,\beta}(\R)$ for every $\beta\in[0,1/2)$ and $\phi_{A_n}\toweak\phi$ weakly in $H^1([0,1])$, as $n\to+\infty$. Furthermore,  for any $1/4<a<b<1$, $\widetilde{r}^A\to-\infty$ uniformly in $[a,b]$ as $A\to-\infty$, while
$$\begin{array}{rcl}
\displaystyle k_\lambda[\widetilde{r}^A]=\displaystyle-\int_0^1(\phi'_A)^2+\lambda^2+\int_0^1\widetilde{r}^A\phi_A^2 & \le & \displaystyle\lambda^2+R\int_0^a\phi_A^2+\int_a^b\widetilde{r}^A\phi_A^2+R\int_b^1\phi_A^2\vspace{3pt}\\
& \le & \displaystyle\lambda^2+R+\int_a^b\widetilde{r}^A\phi_A^2\end{array}$$
for every $A\le0$. Using again the boundedness of the family $(k_\lambda[\widetilde{r}^A])_{A\le0}$, we get
$$\limsup_{A\to-\infty}\int_a^b|\widetilde{r}^A|\phi_A^2<+\infty,$$
whence $\int_a^b\phi^2=0$. Since this holds for every $1/4<a<b<1$ and $\phi\in H^1_{loc}(\R)$ is (at least) continuous, it follows that $\phi=0$ in $[1/4,1]$, and then in $[-3/4,0]$ by $1$-periodicity. In particular, the restriction of $\phi$ to $[0,1/4]$ belongs to $H^1_0([0,1/4])$. Finally, by integrating~\eqref{eq:eigv_pb_phiA} (with $A=A_n$) against any function $v\in\mcc^1_c((0,1/4))$ and taking the limit as $n\to+\infty$, we get:
\[-\int_0^{1/4}v'\phi'+2\lambda\int_0^{1/4}v\phi'+(R+\lambda)^2\int_0^{1/4}v\phi=\ell\int_0^{1/4}v\phi.\]
Therefore, $\phi$ is a weak $H^1_0([0,1/4])$ and then a strong $H^2([0,1/4])\cap H^1_0([0,1/4])$ and then a $\mcc^\infty([0,1/4])$ solution of
$$\left\{\begin{array}{ll}
\phi''+2\lambda\phi'+(R+\lambda^2)\phi=\ell\,\phi & \text{in $[0,1/4]$},\vspace{3pt}\\
\phi\geq 0 & \text{in $[0,1/4]$},\vspace{3pt}\\
\phi(0)=\phi(1/4)=0.& \end{array}\right.$$
But $\|\phi\|_{L^2([0,1/4])}=\|\phi\|_{L^2([0,1])}=1$ (because $\phi=0$ in $[0,1/4]$), whence $\phi>0$ in $(0,1/4)$ from the strong maximum principle. By uniqueness of the principal eigenvalue of this problem, $\ell$ is then equal to the principal eigenvalue $k_{\lambda}^{d}=R-16\pi^2$ of~\eqref{eq:eig_pb_dir}, that is, $\ell=R-16\pi^2$. Since $R-16\pi^2$ does not depend on any sequence of the bounded family $(k_\lambda[\widetilde{r}^A])_{A\le0}$, we conclude that, for every $\lambda\in\R$,
$$k_\lambda[\widetilde{r}^A]\mathop{\longrightarrow}_{A\to-\infty}R-16\pi^2=M.$$
Lastly, pick any $\lambda_0>M^2$. One then has $(R-16\pi^2)/\lambda_0=M/\lambda_0<1/M$. Thus,
$$\frac{k_{\lambda_0}[\widetilde{r}^A]}{\lambda_0}<\frac1M$$
for all $A<0$ negative enough, whence
$$0<c_+[\widetilde{r}^A]=c_-[\widetilde{r}^A]\le\frac{k_{\lambda_0}[\widetilde{r}^A]}{\lambda_0}<\frac1M$$
for all $A<0$ negative enough, by~\eqref{c-=c+} and~\eqref{eq:garfre} (remember also that $k_0[\widetilde{r}^A]>M>0$). Any function $\widetilde{r}:=\widetilde{r}^A$ for $A<0$ negative enough then satisfies~\eqref{ineq:M}. The proof of items~$(i)$-$(ii)$ of Theorem~\ref{thm:no_general_comparison} is thereby complete.
\end{proof}

\begin{proof}[Proof of Theorem~\ref{thm:no_general_comparison}, item $(iii)$]
We assume here that $r$ is constant, whence $r$ is a positive constant (since $0<k_0[r]=r$) and $c_\pm[r]=2\sqrt{r}$. Let $\widetilde{r}\in\mcc^{0,\alpha}(\R)$ be a non-constant $1$-periodic function such that
$$0<k_0[\widetilde{r}]\le k_0[r]=r.$$
The quantity $c_-[\widetilde{r}]=c_+[\widetilde{r}]$ is positive by~\eqref{c-=c+}. Our goal is to show that $c_+[\widetilde{r}]<c_+[r]=2\sqrt{r}$.

First of all, for every $\lambda\in\R$, it follows from~\cite[Theorem~2.2]{Nad10} that
\begin{equation}\label{eq:var_nadin}
k_\lambda[\widetilde{r}]=\min_{\mu\in\mcc^1(\R),\, \mu\text{ is $1$-periodic}}k_0[\widetilde{r}+\lambda^2+(\mu')^2+2\lambda\mu']  
\end{equation}
Taking $\mu\equiv 0$, we find
$$k_{\lambda}[\widetilde{r}]\leq k_0[\widetilde{r}+\lambda^2]=k_0[\widetilde{r}]+\lambda^2\le k_0[r]+\lambda^2=r+\lambda^2.$$
Together with~\eqref{c-=c+} and~\eqref{eq:garfre}, this already shows that $c_+[\widetilde{r}]\le2\sqrt{r}=c_+[r]$. The goal of the next paragraph is to prove that strict inequality holds.

Consider now any $\lambda>0$. By the proof of~\cite[Theorem~2.2]{Nad10} (see in particular~\cite[Proposition 2.6]{Nad10}), the minimum in~\eqref{eq:var_nadin} is actually reached at
\begin{equation}\label{defmum}
\mu_m:=\frac{1}{2}\ln\Big(\frac{\phi}{\psi}\Big),
\end{equation}
where $(\phi,\psi)$ is the unique pair (up to multiplication of each function by positive constants) of $1$-periodic positive $\mcc^{2,\alpha}(\R)$ functions solving the following pair of adjoint equations (recall that an operator and its adjoint have the same principal eigenvalue):
\[
\left\{
\begin{aligned}
  &\phi''+2\lambda\phi'+(\widetilde{r}+\lambda^2)\phi\ =k_{\lambda}[\widetilde{r}]\phi,\\
  &\psi''-2\lambda\psi'+(\widetilde{r}+\lambda^2)\psi=k_{\lambda}[\widetilde{r}]\psi.
\end{aligned}
\right.
\]
If the function $\mu_m$ were constant, the functions $\phi$ and $\psi$ would be identical up to multiplication, and therefore equal without loss of generality. We would then have $-2\lambda\phi'=2\lambda\phi'$ in $\R$, and $\phi$ (and $\psi$) would then be constant, which is impossible since $\widetilde{r}$ is not constant. As a consequence, $\mu_m$ is not constant. Further, by~\cite{Pin95}, the function
$$\nu\mapsto g(\nu):=k_0[\widetilde{r}+\lambda^2+\nu^2+2\lambda\nu]$$
is convex and analytic in the class of continuous $1$-periodic functions $\nu$. For any such functions $\nu$ and $\omega$ with $\omega\not\equiv0$, and for any $t\in\R$, $g(\nu+t\omega)$ is the principal eigenvalue of the operator $\mathcal{L}_0[\widetilde{r}+\lambda^2+(\nu+t\omega)^2+2\lambda(\nu+t\omega)]$. By integrating the equation satisfied by a corresponding principal $1$-periodic eigenfunction $\varphi$ against $1/\varphi$ over $[0,1]$, it follows that
\begin{align*}
  g(\nu+t\omega)&=k_0[\widetilde{r}+\lambda^2+(\nu+t\omega)^2+2\lambda(\nu+t\omega)]\\
  &\ge\int_0^1\big(\widetilde{r}+\lambda^2+(\nu+t\omega)^2+2\lambda(\nu+t\omega)\big),
\end{align*}
whence $\liminf_{t\to\pm\infty}g(\nu+t\omega)/t^2>0$. Hence, the analytic function $t\mapsto g(\nu+t\omega)$ is convex and not affine, so it is strictly convex. As a consequence, $g$ is strictly convex (in the sense that $g((\nu_1+\nu_2)/2)<(g(\nu_1)+g(\nu_2))/2$ for any two distinct continuous $1$-periodic functions $\nu_1$ and $\nu_2$). Therefore, the minimum in~\eqref{eq:var_nadin} is reached only at the function $\mu_m$ defined in~\eqref{defmum}, up to additive constants. Since $\mu_m$ is not constant, we infer that the function $0$ does not minimize~\eqref{eq:var_nadin}, whence
$$k_\lambda[\widetilde{r}]=k_0[\widetilde{r}+\lambda^2+(\mu_m')^2+2\lambda\mu_m']<k_0[\widetilde{r}+\lambda^2]=k_0[\widetilde{r}]+\lambda^2.$$
This implies that $k_{\lambda}[\widetilde{r}]<k_0[r]+\lambda^2$ for every $\lambda>0$, and we conclude that
\[c_+[\widetilde{r}]=\min_{\lambda>0}\frac{k_{\lambda}[\widetilde{r}]}{\lambda}<\min_{\lambda>0}\frac{k_0[r]+\lambda^2}{\lambda}=2\sqrt{k_0[r]}=2\sqrt{r}=c_+[r].\]
This completes the proof of item~$(iii)$ of Theorem~\ref{thm:no_general_comparison}.
\end{proof}

\begin{proof}[Proof of Theorem~\ref{thm:no_general_comparison}, item $(iv)$]
Item~$(iv)$ is a consequence of item~$(iii)$ (where the roles of~$r$ and $\widetilde{r}$ are reversed). Namely, we assume here that~$r$ is non-constant. Since $0<k_0[r]=k_0[k_0[r]]$, viewing $k_0[r]$ as a constant function and applying $(iii)$ with the pair $(r,k_0[r])$ in place of $(\widetilde{r},r)$, we get
\[0<c_+[r]<c_+[k_0[r]]=2\sqrt{k_0[r]}.\]
This proves the first part of item~$(iv)$.

Now, for $\eps>0$ small enough, we have both
\[k_0[r]>k_0[r]-\eps=k_0[k_0[r]-\eps]>0\]
and
\[0<c_+[r]<2\sqrt{k_0[r]-\eps}=c_+[k_0[r]-\eps].\]
Hence, the second part of item~$(iv)$ holds with the positive constant $\widetilde{r}:=k_0[r]-\eps$.
 \end{proof}


\appendix

\section{A condition for local convergence to the positive equilibrium state $p$}\label{s:proof_persistence_general_b}

We here derive a condition more general than $k_0[r;0]>0$ in the case $b=0$, which guarantees local convergence to the $1$-periodic positive steady state $p$ as $t\to+\infty$.

\begin{proposition}\label{ppn:persistence_general_b}
If
\begin{equation}\label{eq:persist2}
k_0\cro{r-\frac{b'}{2}-\frac{b^2}{4};0}>0,
\end{equation}
then~\eqref{eq:persist_cond} holds and $u(t,\cdot)\to p$ as $t\to+\infty$ locally uniformly in $\R$, for any continuous non-negative compactly supported initial condition $u_0\not\equiv0$. In particular, if
$$4\int_0^1\!\!r\,>\,\int_0^1b^2,$$
then~\eqref{eq:persist2} holds.
\end{proposition}

\begin{proof}
As observed in~\cite{BerHamRoq05a} in the case $b=0$, the important quantity for local extinction or local convergence to a positive stationary state is not the periodic principal eigenvalue $k_0[r;b]$; it is rather the limit of Dirichlet eigenvalues in increasing bounded domains. Namely, for $R>0$, let $k_0^d(R)$ be the principal eigenvalue of the Dirichlet eigenproblem
\begin{equation}\label{eq:eigenproblem_dirichlet}
\begin{cases}
  \phi''+b\phi'+r\phi=k_0^d(R)\,\phi&\text{in $[-R,R]$},\\
  \phi>0&\text{in $(-R,R)$},\\
  \phi(-R)=\phi(R)=0,
\end{cases}
\end{equation}
associated with $\mcc^{2,\alpha}([-R,R])$ eigenfunctions $\phi$. By~\cite[Equation 1.13]{BNV94}, as in~\eqref{klambdadir}, there holds
\[k_0^d(R)
=\min_{\psi\in W^{2,1}_{loc}((-R,R)),\,\psi>0\hbox{\small{ in }}(-R,R)}\pth{\mathop{\mathrm{esssup}}_{x\in(-R,R)}\frac{\psi''(x)+b(x)\psi'(x)+r(x)\psi(x)}{\psi(x)}}\]
 for each $R>0$, and the minimum with respect to $\psi$ is reached only by the positive multiples of the principal eigenfunction $\phi$ of~\eqref{eq:eigenproblem_dirichlet}. Hence the function $R\mapsto k_0^d(R)$ is non-decreasing, bounded, and $k_0^d(R)<k_0[r;b]$, since $k_0[r;b]$ is associated with a positive $1$-periodic principal eigenfunction, whose restriction to $[-R,R]$ is therefore still positive in $[-R,R]$. Thus, there exists a limit
\[k_0^d(\infty):=\lim_{R\to+\infty}k_0^d(R)\ \in\R,\]
such that
\begin{equation}\label{ineqk0}
k_0^d(\infty)\le k_0[r;b].
\end{equation}
Let us prove that
\begin{equation}\label{eq:equality_kdir_kper}
k_0^d(\infty)=k_0\!\cro{r-\frac{b'}{2}-\frac{b^2}{4};0},
\end{equation}
where, using the notations of the introduction, $k_0\big[r-\frac{b'}{2}-\frac{b^2}{4};0\big]$ is the periodic principal eigenvalue associated with the operator~$\mathcal{T}$ defined by
\[\mathcal{T}\,:\,\phi\mapsto\phi''+\pth{r-\frac{b'}{2}-\frac{b^2}{4}}\phi.\]
To show~\eqref{eq:equality_kdir_kper}, let, for each $R>0$, $\phi_R\in\mcc^{2,\alpha}([-R,R])$ be a positive principal eigenfunction of~\eqref{eq:eigenproblem_dirichlet}, and set $\psi_R=\phi_R\,e^{B/2}$, with $B'=b$. As in the proof of Proposition~\ref{prop:k0neq}, the $\mcc^{2,\alpha}([-R,R])$ function $\psi_R$ solves $\mathcal{T}\psi_R=k_0^d(R)\,\psi_R$ in $[-R,R]$, with $\psi_R>0$ in $(-R,R)$ and ${\psi_R(\pm R)=0}$. Thus, $k_0^d(R)$ is the Dirichlet principal eigenvalue of~$\mathcal{T}$ in $[-R,R]$. Since the operator $\mathcal{T}$ is self-adjoint and has periodic coefficients, it now follows from \cite[Lemma~3.6]{BerHamRoq05a} that $k_0^d(R)$ converges, as $R\to+\infty$, to the periodic principal eigenvalue of~$\mathcal{T}$, which is $k_0\big[r-\frac{b'}{2}-\frac{b^2}{4};0\big]$. To sum up, we have
\begin{equation}\label{eq:limit}
k_0^d(R)\mathop{\longrightarrow}_{R\to+\infty}k_0\cro{r-\frac{b'}{2}-\frac{b^2}{4};0}.
\end{equation}
Therefore~\eqref{eq:equality_kdir_kper} holds.

Thus, a comparison between $k_0\big[r-\frac{b'}{2}-\frac{b^2}{4};0\big]$ and $0$ is equivalent to a comparison between $k_0^d(\infty)$ and $0$. The proof of local convergence of the solutions~$u$ to~\eqref{eq:KPP} with continuous non-negative compactly supported initial conditions $u_0\not\equiv0$ to a stationary state if $k_0\big[r-\frac{b'}{2}-\frac{b^2}{4};0\big]>0$ is the same as the proof of~\cite[Theorem 2.6]{BerHamRoq05a} (replacing the notation there $\lambda_1$ by our notation $-k_0^d(\infty)$). In particular, the positivity of $k_0\big[r-\frac{b'}{2}-\frac{b^2}{4};0\big]$ rules out the extinction of these solutions. Notice also that $k_0\big[r-\frac{b'}{2}-\frac{b^2}{4};0\big]>0$ implies directly $k_0[r;b]>0$, by~\eqref{ineqk0} and~\eqref{eq:limit}.

Lastly, let $\psi\in\mcc^{2,\alpha}(\R)$ be a $1$-periodic positive principal eigenfunction of $\mathcal{T}$, that is, $\mathcal{T}\psi=k_0\big[r-\frac{b'}{2}-\frac{b^2}{4};0\big]\,\psi$ in $\R$. By integrating this equation against $1/\psi$ over $[0,1]$, one gets that
$$k_0\Big[r-\frac{b'}{2}-\frac{b^2}{4};0\Big]=\int_0^1\frac{(\psi')^2}{\psi^2}+\int_0^1\Big(r-\frac{b'}{2}-\frac{b^2}{4}\Big)\ge\int_0^1r-\frac14\int_0^1b^2.$$
Therefore, if $\displaystyle 4\int_0^1r>\int_0^1b^2$, then $k_0\big[r-\frac{b'}{2}-\frac{b^2}{4};0\big]>0$.\end{proof}


\section{Construction of the principal eigenfunctions in item~$(ii)$ of Theorem~\ref{thm:equality}}\label{s:construction_phi}

In the proof of item~$(ii)$ of Theorem~\ref{thm:equality}, in order to show that $k_0[r;b]=k_0[r;-b]$ when $r=\beta+\gamma b$, $\beta\in\R$, $\gamma\in\R^*$, and $b\in\mcc^{0,\alpha}(\R)$ is $1$-periodic, we gave in~\eqref{defvarphi} an explicit expression for a principal $1$-periodic eigenfunction $\varphi$ solving~\eqref{eq:varphi} and associated with $\mcl_{0}[r;-b]$, using a principal $1$-periodic eigenfunction $\phi$ solving~\eqref{eq:phi+} and associated with $\mcl_{0}[r;b]$. As this expression is somewhat complicated, we explain its origin in this subsection. We assume for conciseness that $\beta=0$, so $r=\gamma b$.

For the moment, we do not consider the periodicity conditions and simply look for the form of the general solution of \eqref{eq:varphi} in $\R$, or equivalently, the general solution $\mu\in\mcc^2(\R)$ of $\T[\mu]=0$ in $\R$, with
$$\T : \mu \mapsto \mu'' -  b\mu' + (\gamma b - k_0[\gamma b ;b]) \mu.$$ 

We begin with the construction of a particular solution $y_1\in\mcc^2(\R)$ of $\T[y_1]=0$ in $\R$, which we seek in the form $y_1=m\,e^B$, with $B'=b$, $B(0)=0$, and with a function $m\in\mcc^2(\R)$ to be determined. With this ansatz, we have 
$$\T[y_1]=e^B\big(m'' + b m' +(\gamma b +b' - k_0[\gamma b ;b] ) m \big).$$
Setting a new operator
$$\widetilde \T : \mu \mapsto \mu'' + b \mu' + (\gamma b +b' - k_0[\gamma b ;b]) \mu,$$
acting on $\mcc^2(\R)$ functions $\mu$, we observe that $\widetilde \T [\phi]= b' \phi$. Furthermore, since $b$ is at least of class $\mcc^1(\R)$ and $\phi$ is at least of class $\mcc^2(\R)$, it follows from~\eqref{eq:phi+} that $\phi$ is actually at least of class $\mcc^3(\R)$. Hence, differentiating~\eqref{eq:phi+}, we get
$\widetilde \T [\phi']=-\gamma \, b' \, \phi$. As a consequence, by choosing
$$m:=\phi'+\gamma \, \phi\ \in\mcc^2(\R),$$
we have $\T[y_1]=e^B\widetilde \T[m]=0$ in $\R$. Using Lemma~\ref{lem:sign_phi'+phi}, as in the beginning of the proof of item~$(ii)$ of Theorem~\ref{thm:equality} in Section~\ref{s:proof_equality}, we note that $m$ has a constant strict sign in~$\R$. Therefore, since
$$b=\frac{k_0[\gamma b;b] \phi -\phi''}{m}$$
by~\eqref{eq:phi+}, it follows that the solution $y_1=me^B$ of $\T[y_1]=0$ can be rewritten as
\begin{align*}
  y_1(x)& = m(x)\exp \lp \int_0^x  b(s) \, ds \rp\\
  & = m(0)\exp \lp \int_0^x  \Big(b(s) + \frac{m'(s)}{m(s)}\Big)\, ds \rp = m(0) \exp(g(x)),  
\end{align*}
with
$$g(x) := \int_0^x \frac{\gamma\, \phi'(s) + k_0[\gamma \, b ;b] \phi(s)}{\phi'(s) + \gamma\, \phi(s)} \, ds.$$

Now, with this particular strictly signed solution $y_1=m(0)\,e^g$ of $\T[y_1]=0$ in hand, still forgetting the periodicity properties (actually, the function $y_1$ is not periodic if the mean of $b$ is not zero), the general solution $\mu$ of 
$\T[\mu]=0$ in $\R$, can be written from \cite[Section 2.2.2]{PolZai17} as
\[\mu(x) = y_1(x) \left( C_1 + C_2 \int_0^x \frac{e^{B(z)}}{y_1^2(z)}\, dz \right)\!,
\]
where $C_1$, $C_2$ are two constants. Since
$$\frac{e^B}{y_1}=\frac{1}{m}=\frac{1}{\phi'+\gamma\phi}\ \hbox{ and }\ y_1=m(0)\,e^g$$
with $m(0)\neq0$, up to a modification of the constants $C_1$ and $C_2$, we can then write the general solution $\mu\in\mcc^2(\R)$ of $\T[\mu]=0$ in $\R$ as:
\[
\mu(x) = e^{g(x)} \left( C_1 + C_2 \int_0^x \frac{e^{-g(z)}}{\phi'(z) + \gamma\, \phi(z)}\, dz \right)=e^{g(x)}\,\big(C_1+C_2h(x)\big),
\]
where
\[h(x):= \int_0^x \frac{e^{-g(z)}}{\phi'(z) + \gamma\, \phi(z)}\, dz=\int_0^x \frac{e^{-g(z)}}{m(z)}\, dz.\]

Finally, observing that $h(1)\neq0$ (because $m$ has a constant strict sign in $\R$), and setting
$$C_1:=1\ \hbox{ and }C_2:=\frac{e^{-g(1)}-1}{h(1)},$$
we obtain a particular solution $\mu$ of $\T[\mu]=0$ in $\R$, with $\mu(0)=\mu(1)=1$. Furthermore, the function $C_1+C_2h=\mu\,e^{-g}$ is monotone in $\R$, and takes positive values at~$0$ and~$1$, whence $\mu$ is positive in $[0,1]$. Lastly, as in the proof of item~$(ii)$ of Theorem~\ref{thm:equality} in Section~\ref{s:proof_equality}, we can check that $\mu'(0)=\mu'(1)$. As a consequence, with that choice of constants $(C_1,C_2)$, the function $\varphi:=\mu\in\mcc^2(\R)$ is a $1$-periodic positive solution of~\eqref{eq:varphi}.

\section*{Acknowledgments} 
N.B., F.H. and L.R. were supported by the ANR project ReaCh, {ANR-23-CE40-0023-01}.
N.B. was supported by the Chaire Modélisation Mathématique et Biodiversité (École Polytechnique, Muséum national d’Histoire naturelle, Fondation de l’École Polytechnique, VEOLIA Environnement). 

\bibliographystyle{plain}
\footnotesize{\bibliography{biblio_clean}}

\end{document}